\documentclass[12pt]{amsart}
\usepackage[left=2.5cm,right=2.5cm, top=2.5cm, bottom=2.5cm]{geometry}
\usepackage{amsmath,amsfonts,amssymb}
\usepackage{color} 
\usepackage{enumerate,graphics, graphicx}
\usepackage{enumitem}
\usepackage{tikz}
\usepackage[all,cmtip]{xy}
	\usetikzlibrary{arrows}
\usepackage[pdftex,bookmarks=true]{hyperref}
\usepackage{verbatim}
\usepackage{caption}
\usepackage{subcaption}

\newtheorem{theorem}{Theorem}[section]
\newtheorem{lemma}[theorem]{Lemma}

\newtheorem{proposition}[theorem]{Proposition}
\newtheorem{prop}[theorem]{Proposition}
\newtheorem{corollary}[theorem]{Corollary}

\theoremstyle{definition}
\newtheorem{defn}[theorem]{Definition}
\newtheorem{definition}[theorem]{Definition}
\newtheorem{eg}[theorem]{Example}

\newtheorem{question}[theorem]{Question}
\newtheorem{prob}[theorem]{Problem}
\newtheorem{rmk}[theorem]{Remark}
\newtheorem{notation}[theorem]{Notation}

\newcommand{\Q}{{\mathbb Q}}

\newcommand{\R}{\mathbb R}

\newcommand{\calC}{\mathcal{C}}
\newcommand{\calR}{\mathcal{R}}
\newcommand{\calS}{\mathcal{S}}

\newcommand{\St}{{S}}
\newcommand{\invtPoly}{\mathcal{P}}

\newcommand{\Rnn}{\mathbb{R}_{\geq 0}}
\DeclareMathOperator{\im}{im}

\def\rla{\rightleftarrows}
\def\been{\begin{enumerate}}
\def\enen{\end{enumerate}}

\def\SS{\mathcal S}
\def\CC{\mathcal C}
\def\RR{\mathcal R}

\def\II{\mathcal I}
\def\OO{\mathcal O}

\usepackage{fancybox}
\usepackage{todonotes}
\usepackage{menukeys}
\newcounter{todocounter}

\presetkeys{todonotes}{inline, color=blue!30}{}

\usepackage{natbib}
\usepackage{url}

\title[Joining Reaction Networks]{Joining and decomposing reaction networks}

\author[Gross]{Elizabeth Gross}
\address{University of Hawai`i at M\={a}noa}
\author[Harrington]{Heather Harrington}
\address{University of Oxford}
\author[Meshkat]{Nicolette Meshkat}
\address{Santa Clara University}
\author[Shiu]{Anne Shiu}
\address{Texas A\&M University}
\date{14 August 2019}

\begin{document}
\maketitle

\begin{abstract}
In systems and synthetic biology,
much research has focused on the behavior and design of single pathways, 
while, more recently, experimental efforts have focused on how cross-talk (coupling two or more pathways) or inhibiting molecular function (isolating one part of the pathway) affects systems-level behavior.  However, the theory for tackling these larger systems in general has lagged behind.  Here, we analyze how joining networks (e.g., cross-talk) or decomposing networks (e.g., inhibition or knock-outs) affects three properties that reaction networks may possess---identifiability (recoverability of parameter values from data), steady-state invariants (relationships among species concentrations at steady state, used in model selection), and multistationarity (capacity for multiple steady states, which correspond to multiple cell decisions).  Specifically, we prove results that clarify, for a network obtained by joining two smaller networks, how 
properties of the smaller networks can be inferred from or can imply 
similar properties of the original network.  
Our proofs use techniques from computational algebraic geometry, including elimination theory and differential algebra. 


\end{abstract}

\maketitle



\section{Introduction} \label{sec:intro}

Cells transmit information 
via molecular interactions which are complicated and numerous: a typical eukaryotic cell contains approximately $8\times10^9$ molecules.
Understanding the function and behavior of such a large number of molecules is challenging and often intractable. Therefore, much effort in the field of systems biology focuses on first understanding and predicting the behavior of smaller sets of interacting molecular species, called signaling pathways. 
Advances in experimental technology have enabled the possibility of measuring more species, prompting questions about what happens when two or more specific pathways interact \citep{donato2013}.   This problem of predicting the effect of joining pathways is the focus of our work.

Whenever two or more pathway models are combined, it is reasonable to expect that some model properties of the larger model may be inferred predictably from properties of the component models. Within this context, our work focuses on three important properties of pathway models:
{\em identifiability}, whether the parameter values can be determined from data, {\em steady-state invariants}, which characterize a model and provide a framework for hypothesis testing with limited data, and {\em multistationarity}, which is the capacity for multiple positive steady states. We prove results on how these properties are affected when we combine two or more models.
We consider, first, linear models, and then extend our results, where possible, to nonlinear models. 

A biological example to motivate our study is signaling in apoptosis (programmed cell death). 
Activation of the death signal can be initiated by either the intrinsic pathway (via stress) or the extrinsic pathway (via ligand-receptor binding). Mathematical models of each pathway have been developed \citep{eissing,legewie}, and analyses of these models have revealed that both pathways have the capacity for two steady-states, which correspond to a cell-death state and a cell-alive 
 state~\citep{bagci,eissing,legewie,ho2010}, meaning that the models are multistationary.
Analyses of cell death models have also focused on identifiability \citep{eydgahi2013} and steady-state invariants \citep{ho2010}.
 Since models by \citet{eissing} and \citet{legewie}, additional models have been constructed with a focus on the molecular network between the intrinsic and extrinsic pathways at the mitochondrial membrane \citep{bagci,albeck,cui} as well as joining both pathways into a single model \citep{harrington2008,fussenegger}.  However, predicting how joining pathways affects cell death checkpoints and other model properties is difficult.  Pursuing this question for general pathway networks is similar in some respects 
to analyzing retroactivity and modules within a larger network~\citep{modular,menon2016}. 

In this work, we are interested in signaling pathway models that describe molecular interactions via biochemical reactions.  In particular, we will study chemical reaction networks, directed graphs in which the nodes are molecular complexes and the edges are reactions weighted by rate constants (parameters). Under the assumption of mass-action kinetics, each reaction graph gives rise to a system of polynomial differential equations. Thus, in essence, we are interested in how this polynomial system of differential equations changes as we construct larger networks from smaller ones.  Since our emphasis is on the structure of the equations, not the value of the parameters, our analysis focuses on properties that hold in general. 

Reaction networks can be joined naturally in various ways; two such ways are shown in Figure~\ref{fig:glue}. 
 As shown in Figure~\ref{fig:glue}A, one way we can {\em glue} together 
 two networks  $X$ and $Y$ is via a new or shared edge. Networks obtained by gluing over new or shared edges arise naturally when considering linear compartmental models and are central to Section~3.    
Another way to glue together $X$ and $Y$ is via a shared node (Figure~\ref{fig:glue}B); such gluing allows us to investigate {\em cross-talk}, interactions between signaling pathways $X$ and $Y$ that have at least one shared molecule.
Currently, cross-talk is an active area of research in biology, especially for predicting the effects of drug targets on cells. Networks obtained by gluing over shared nodes are analyzed in terms of their steady-state invariants in Section~4.

\begin{figure}[h!]
\includegraphics{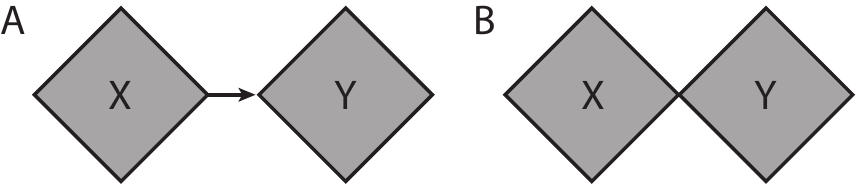}
\caption{
(A) Gluing two networks via a new edge. Biologically, this may correspond to distinct networks of the same pathway. (B) Gluing two motifs that have a shared species (node). Biologically, this may be called cross-talk.}
\label{fig:glue}
\end{figure}

  Our main results are as follows.   Consider joining two networks
  $X$ and $Y$ to obtain a new network $Z$.
  We show that, under certain hypotheses,
  if $X$ and $Y$ are identifiable, then so is $Z$
  (Theorem~\ref{thm:join-new-rxn-iden}).  Similarly, for certain
  monomolecular networks, not only does identifiability of $X$ and
  $Y$ imply identifiability of $Z$, but also identifiability of $Z$
  and $X$ implies that $Y$ is identifiable
  (Theorems~\ref{thm:onewayflow} and~\ref{thm:super-i-o-iff}).
Also, we clarify how the steady-state invariants of $Z$, after projecting them to involve only species and reactions in $N_i$, are related to the steady-state invariants of $N_i$.  
We give conditions when the projected steady-state invariants yield all invariants of $N_i$ (Theorems~\ref{thm:mono-single-species} and~\ref{thm:mono-reaction}), and when 
the steady-state invariants of $X$ and $Y$ can together recover the
steady-state invariants of $Z$ (Theorem~\ref{thm:mono-glue-species-invt}).

The outline of our work is as follows.  
Section~\ref{sec:background} introduces the background and definitions.  
Next, Sections~\ref{sec:iden}, \ref{sec:elim}, and~\ref{sec:mss}
each correspond to a property of interest: identifiability, steady-state invariants, and multistationarity (respectively).
The proofs of our results rely on techniques from computational algebraic geometry, such as elimination theory and differential algebra;
indeed, algebraic tools are increasingly used in the analyses of reaction networks (see the survey by ~\citet{invitation}). 
Finally, a discussion appears in Section~\ref{sec:discussion}.


\section{Background} \label{sec:background}

Valuable information may be obtained by translating a chemical reaction network into a system of differential equations. In our setting, we form a polynomial dynamical system which is amenable to algebraic analysis described in the subsequent 
 sections. 
First, we begin with an example of a {\em chemical reaction}: $A+B \to 3A + C$, where $A$, $B$, and $C$ are chemical {\em species}. 
These species could represent various proteins modifying one another.
%
In this reaction, the {\em reactant} forms the left hand side of the reaction (one species $A$ and one of $B$), which react to form the {\em product} (three $A$ and one $C$).

We follow convention and denote concentrations of the species by lower case $x_{A},$ $x_{B}$, and
$x_{C}$, which
will change in time as the reaction occurs.  Here, we assume {\em
mass-action kinetics}, that is, species $A$ and $B$ react at a rate proportional to the
product of their concentrations, where the proportionality constant is the {\em reaction rate
constant} $\kappa$.  Noting that the reaction yields a net change of two units in
the amount of $A$, we obtain the differential equation $\frac{d}{dt}x_{A}=2\kappa x_{A}x_{B}$, where $t$ is time.
The other two equations arise similarly:
$\frac{d}{dt}x_{B} =-\kappa x_{A}x_{B}$ and 
$\frac{d}{dt} x_{C}=\kappa x_{A}x_{B}$.

 A {\em chemical reaction network}
consists of finitely many reactions (see Definition~\ref{def:crn} below).  
The mass-action differential equations that a network defines are a 
sum of the monomial contributions from the reactants of each 
chemical reaction in the network; these 
differential equations 
will be defined in equation~(\ref{eq:ODE-mass-action}).

\subsection{Chemical reaction networks}
We now provide precise definitions.  

\begin{definition} \label{def:crn}
A {\em chemical reaction network} $G=(\SS,\CC,\RR)$
consists of three finite sets $\SS, \CC$, and $\RR$.
\begin{enumerate}
\item A set of chemical {\em species} $\SS = \{A_1,A_2,\dots, A_n \}$, where $n \in \mathbb{N}$ denotes the number of species.
\item A set  $\CC = \{y_1, y_2, \dots, y_p\}$ of {\em complexes} (finite nonnegative-integer combinations of the species), where $p \in \mathbb{N}$ denotes the number of complexes. 
\item A set of {\em reactions}, ordered pairs of the complexes: $\RR \subseteq (\CC \times \CC) \setminus \{ (y_i,y_i) \mid y_i \in \CC\}$. 
\end{enumerate}
\end{definition}
Throughout this work, the integer unknown~$r$ denotes the number of
reactions.  A {\em subnetwork} of a network $G = (\SS, \CC, \RR)$ is a network $\widetilde{G} = (\widetilde{\SS}, \widetilde{\CC}, \widetilde{\RR})$ with $\widetilde{\RR} \subseteq {\RR}$.

\noindent


We also make a simplifying assumption: every complex in $\CC$ must appear in at least one reaction in $\RR$, and every species in $\SS$ must appear in at least one complex in $\CC$.  This assumption does not restrict the class of networks we can study, just how they are represented.

A network can be viewed as a directed graph whose nodes are complexes and whose edges correspond to the reactions. Like for all network analysis, properties of the connectedness of the graph can be useful. 
A reaction $y_i \to y_j$ is {\em reversible} if it is bi-directional, i.e., the reverse reaction $y_j \to y_i$ is also in $\RR$; these reactions are depicted by $y_i \rightleftharpoons y_j$.

Writing the $i$-th complex as $y_{i1} A_1 + y_{i2} A_2 + \cdots + y_{in}A_n$ 
(where $y_{ij} \in \mathbb{Z}_{\geq 0}$, for $j=1,2,\dots,n$, are the {\em stoichiometric coefficients}), 
we introduce the following monomial:
$$ x^{y_i} \,\,\, := \,\,\, x_1^{y_{i1}} x_2^{y_{i2}} \cdots  x_n^{y_{in}}~. $$
(By convention, the {\em zero complex} yields the monomial $x^{(0,\dots,0)}=1$.)  
For example, the two complexes in the reaction $A+B \to 3A + C$ considered earlier give rise to 
the monomials $x_{A}x_{B}$ and $x^3_A x_C$, which determine the vectors 
$y_1=(1,1,0)$ and $y_2=(3,0,1)$.  
These vectors define the rows of a $p \times n$-matrix of nonnegative integers,
which we denote by $Y=(y_{ij})$.
Next, the unknowns $x_1,x_2,\ldots,x_n$ represent the
concentrations of the $n$ species in the network,
and we regard them as functions $x_i(t)$ of time $t$.

We distinguish between monomolecular complexes (e.g., $A$ or $B$), bimolecular complexes (e.g., $2A$ or $A+B$), and others (e.g., $0$ or $A+2B$), as follows.  
A complex  $y_{i1} A_1 + y_{i2} A_2 + \cdots + y_{in}A_n$ is {\em monomolecular} if exactly one stoichiometric coefficient $y_{ij}$ equals 1, and all other $y_{ik}$'s are 0.  A complex $y_{i1} A_1 + y_{i2} A_2 + \cdots + y_{in}A_n$ is {\em at-most-bimolecular} if the sum of the stoichiometric coefficients $y_{ij}$ is at most 2.  A reaction network is itself {\em monomolecular} (respectively, {\em at-most-bimolecular}) if all its complexes are monomolecular or the zero complex (respectively, all its complexes are at-most-bimolecular). 
The reaction systems arising from monomolecular networks are known as linear compartmental models (see \S \ref{subsec:monomolecular}).

For a reaction $y_i \to y_j$ from the $i$-th complex to the $j$-th
complex, the {\em reaction vector}
 $y_j-y_i$ encodes the
net change in each species that results when the reaction takes
place.  The {\em stoichiometric matrix} 
$\Gamma$ is the $n \times r$ matrix whose $k$-th column 
is the reaction vector of the $k$-th reaction 
i.e., it is the vector $y_j - y_i$ if $k$ indexes the 
reaction $y_i \to y_j$.

We associate to each reaction 
a positive parameter $\kappa_{ij}$, the {\em rate constant} of the
reaction.  In this article, we will treat the rate constants $\kappa_{ij}$ as positive
unknowns in order to analyze the entire family of dynamical systems
that arise from a given network as the $\kappa_{ij}$'s vary.

\subsection{Chemical reaction systems}


The {\em reaction kinetics system} 
defined by a reaction network $G$ and reaction rate function $R:\Rnn^n \to \R^r$ is given by the following system of ODEs:
\begin{align} \label{eq:ODE}
\frac{dx}{dt} ~ = ~ \Gamma \cdot R(x)~.
\end{align}
A {\em steady state} of a reaction kinetics system~\eqref{eq:ODE} is a nonnegative concentration vector $x^* \in \Rnn^n$ at which the ODEs~\eqref{eq:ODE}  vanish: $ \Gamma \cdot R(x^*)=0$. 

For {\em mass-action kinetics}, which is the setting of this paper, the coordinates of $R$ are
$ R_k(x)=  \kappa_{ij} x^{y_i}$, 
 if $k$ indexes the reaction $y_i \to y_j$.  
A {\em chemical reaction system} refers to the 
system of differential equations (\ref{eq:ODE}) arising from a specific chemical reaction
network $(\SS, \CC, \RR)$ and a choice of  rate constants $(\kappa_{ij}) \in
\mathbb{R}^{r}_{>0}$ (recall that $r$ denotes the number of
reactions) where the reaction rate function $R$ is that of mass-action
kinetics.  Specifically, the mass-action ODEs are: 
\begin{align} \label{eq:ODE-mass-action}
	\frac{dx}{dt} \quad = \quad \sum_{ y_i \to y_j~ {\rm is~in~} \RR} \kappa_{ij} x^{y_i}(y_j - y_i) \quad =: \quad f_{\kappa}(x)~.
\end{align}

The {\em stoichiometric subspace} is the vector subspace of
$\mathbb{R}^n$ spanned by the reaction vectors
$y_j-y_i$, and we will denote this
space by $\St$: 
\begin{equation} \label{eq:stoic_subs}
  \St~:=~ \mathbb{R} \{ y_j-y_i \mid  y_i \to y_j~ {\rm is~in~} \RR \}~.
\end{equation}
  Note that in the setting of (\ref{eq:ODE}), one has $\St = \im(\Gamma)$.
For the network consisting of the single reaction $A+B \to 3A + C$, we have $y_2-y_1 =(2,-1,1)$, which
means 
that with each occurrence of the reaction, two units of $A$ and one of $C$ are
produced, while one unit of $B$ is consumed.  This vector $(2,-1,1)$ spans the
stoichiometric subspace $\St$ for the network. 
Note that the  vector $\frac{d x}{dt}$ in  (\ref{eq:ODE}) lies in
$\St$ for all time $t$.   
In fact, a trajectory $x(t)$ beginning at a positive vector $x(0)=x^0 \in
\R^n_{>0}$ remains in the {\em stoichiometric compatibility class},
which we denote by
\begin{align}\label{eqn:invtPoly}
\invtPoly~:=~(x^0+\St) \cap \mathbb{R}^n_{\geq 0}~, 
\end{align}
for all positive time.  In other words, $\invtPoly$ is forward-invariant with
respect to the dynamics~(\ref{eq:ODE}).    

\subsection{Combining networks} \label{subsec:glue}
Here we introduce operations that allow two or more networks to be `glued' together to form a single network.  These operations encompass many natural operations that arise in biological modeling, for instance, connecting two networks by a one-way flow, or extending a model to include additional pathways. 
The aim of this work is to investigate how these operations affect three properties of networks: identifiability, steady-state invariants, and multistationarity.  

\begin{definition} \label{def:union}
The {\em union} of reaction networks $N_1=(\calS_1, \calC_1, \calR_1)$ and $N_2=(\calS_2, \calC_2, \calR_2)$ is 
	\begin{align*}
	N_1 \cup N_2 ~:=~ \left( \calS_1 \cup \calS_2, ~ \calC_1 \cup \calC_2,~ \calR_1 \cup \calR_2 \right)~.
	\end{align*}
The union of finitely many reaction networks $N_i$ is defined similarly.	
\end{definition}

Next, we classify the union $N_1 \cup N_2$ according to whether their respective sets of complexes (or reactions) of $N_i$ are disjoint.  The possible relationships among these sets are constrained by the following implications:
	\begin{align*}
	\calS_1 \cap \calS_2 = \emptyset \quad \Rightarrow \quad 
	\calC_1 \cap \calC_2= \emptyset {\rm~~or~~}  \calC_1 \cap \calC_2= \{0\} 
		 \quad \Rightarrow \quad
	 \calR_1 \cap \calR_2= \emptyset ~.
	\end{align*}
If the two species sets are disjoint ($\calS_1 \cap \calS_2 = \emptyset$), then the networks $N_1$ and $N_2$ are completely disjoint, so analyzing their union is equivalent to analyzing $N_1$ and $N_2$ separately.  Thus, we are interested in the three remaining cases:
\begin{definition} \label{def:glue}
The union $N_1 \cup N_2$ of $N_1=(\calS_1, \calC_1, \calR_1)$ and $N_2=(\calS_2, \calC_2, \calR_2)$ is formed by:
	\begin{enumerate}
	\item {\em gluing complex-disjoint networks} if $\calS_1 \cap \calS_2 \neq \emptyset$ and the two networks have no complex in common except possibly the zero complex, i.e.,  
		$\calC_1 \cap \calC_2  \subseteq \{0\}$ (and thus $\calR_1 \cap \calR_2= \emptyset$),
	\item {\em gluing over complexes} if the two networks have at least one non-zero complex in common (i.e., $\calC_1 \cap \calC_2  \nsubseteq \{0\}$)
		but no reactions in common (i.e., $\calR_1 \cap \calR_2= \emptyset$),
	\item {\em gluing over reactions} if the two networks have at least one reaction in common (i.e., $\calR_1 \cap \calR_2 \neq \emptyset$).
	\end{enumerate}
\end{definition}

\noindent
{\bf Notation.} We will denote the species of $N_1 \cup N_2$ as ${\bf x}=x_1, \ldots, x_n$, and the species of $N_1$ and $N_2$ as 
${\bf x}(1)=\{x_1,\dots, x_j\}$ and ${\bf x}(2)=\{x_{k},\dots, x_n\}$, 
respectively. 
Here, $k \leq j$, because the species sets overlap.  
We let ${\bf \kappa}(1)$ and ${\bf \kappa}(2)$ be the rate constants of the reactions in $\calR_1$ and $\calR_2$, respectively, and we let ${\bf \kappa}={\bf \kappa}(1) \cup {\bf \kappa}(2)$ denote the  rate constants of $N$.

\begin{rmk} \label{rmk:monomolecular}
If networks $N_1$ and $N_2$ are monomolecular, then they are complex-disjoint if and only if they are species-disjoint ($\mathcal S_1 \cap \mathcal S_2 = \emptyset$).  Thus, we can not glue complex-disjoint 
networks that are monomolecular.
\end{rmk}

We introduce more operations, in which $N_1$ and $N_2$ may have disjoint species sets:
\begin{definition} \label{def:glue-new-rxn}
Consider networks $N_1=(\calS_1, \calC_1, \calR_1)$ and $N_2=(\calS_2, \calC_2, \calR_2)$.
	\begin{enumerate}
	\item Let $\{y \to y'\}$ denote a network that consists of a single reaction that is not in $\calR_1 \cup \calR_2$ and for which $y \in \calC_1$ and $y' \in \calC_2$.
The network obtained by {\em joining $N_1$ and $N_2$ by a new reaction} $y \to y'$ is:
	\begin{align*}
	N_1 \cup N_2 \cup \{y \to y'\}~.
	\end{align*}
	\item Let $\calR'$ and $\calR''$ be sets of reactions for which $\calR' \subseteq (\calR_1 \cup \calR_2)$ and $\calR'' \cap ( \calR_1 \cup \calR_2)  = \emptyset$, and every 
reaction $y \to y'$ in $\calR''$ satisfies $y \in \calC_1$ and $y' \in \calC_2$. Let  
$N_3$ 
denote the network that consists of the reactions in $\calR''$. The network obtained by {\em joining $N_1$ and $N_2$ by replacing reactions $\calR'$ by $\calR''$} is: 
	 		\begin{align*}
	(\calS_1, \calC_1, \calR_1\setminus \calR') ~ \cup ~ (\calS_2, \calC_2, \calR_2\setminus \calR') ~ \cup ~ N_3~. 
	\end{align*}

	\end{enumerate}
\end{definition}

\noindent
Joining by a new reaction, in Definition~\ref{def:glue-new-rxn}(1), adds a one-way flow from one network to another.
As for replacing reactions, in Definition~\ref{def:glue-new-rxn}(2), we describe an instance of this.  Suppose that a large network is formed by two subnetworks $M_1$ and $M_2$, plus a reaction $X \to Y$ from $M_1$ to $M_2$.  Then, to study each subnetwork separately, we might consider $N_1=M_1 \cup \{X \to 0\}$ and $N_2=M_2 \cup \{0 \to Y\}$.  Later, when we want to put these two networks together, we {\em join $N_1$ and $N_2$ by replacing reactions $\{X \to 0 \to Y\}$ by $\{X \to Y\}$.}

\begin{eg} \label{ex:join}
Consider the following networks: 
	\[
	N_1 ~=~ \{ 0 \to A \to B \}~, \quad 
	N_2 ~=~ \{ A \to B \leftrightarrow C \}~, \quad {\rm and} \quad
	N_3 ~=~ \{ C \leftrightarrow D \}~.
	\]
The network $N_1 \cup N_2$ formed by gluing over the shared reaction $A \to B$ is $\{ 0 \to A \to B \leftrightarrow C\}$.  Also, the network obtained by joining $N_1$ and $ N_3$ by a new reaction $B \to C$ is $\{ 0 \to A \to B \to  C \leftrightarrow D\}$.
\end{eg}

\begin{rmk} \label{rmk:subnetwork}
Using the definitions above and recalling our assumption that networks include only those species or complexes that take part in reactions, we see that a network $N$ is a {\em subnetwork} of $G$ if there exists a network $N'$ for which $G=N \cup N'$.  In this case, to obtain the mass-action ODEs~\eqref{eq:ODE-mass-action} for $N$ from those of $G$, simply set all  rate constants to zero for those reactions not in $N$.  
As for the ODEs obtained by gluing networks as in Definition~\ref{def:glue}, 
we clarify them in Lemma~\ref{lem:glue-ODEs}.
\end{rmk}

The next result follows from the fact that the mass-action ODEs are a sum over reactions.
\begin{lemma} \label{lem:glue-ODEs}
Consider networks $N_1=(\calS_1, \calC_1, \calR_1)$ 
and  $N_2=(\calS_2, \calC_2, \calR_2)$, and denote their 
 mass-action ODEs~\eqref{eq:ODE-mass-action} by, respectively,
  $ dx/dt= f$ and $ dx/dt= g$.  Define $f_i:=0$ (respectively, $g_i:=0$)
  for species $i \in \calS_2 \setminus \calS_1$ (respectively, $i \in \calS_1 \setminus \calS_2$).
Let $N=N_1 \cup N_2$ be the reaction network obtained by gluing $N_1$ and $N_2$. Then the mass-action ODEs for $N$ are given by:
\begin{enumerate}
	\item  $dx/dt = f + g$, if $\calR_1 \cap \calR_2 = \emptyset$ (i.e., gluing complex-disjoint networks or over complexes).
	\item  
	$ dx/dt = f + \widetilde g$,   if $\calR_1 \cap \calR_2 \neq \emptyset$ (i.e., gluing over reactions), where
$ dx/dt = \widetilde g$ denotes the mass-action ODEs for the subnetwork of $N_2$ comprising only reactions in $\calR_2 \setminus \calR_1$.
\end{enumerate}
\end{lemma} 

\begin{rmk} \label{rmk:johnston}
A related approach to gluing networks, introduced by ~\citet{translated},
involves ``translating'' some of the complexes in such a way that 
the ``translated'' networks 
 (taken with certain general kinetics) define the same dynamical systems
 as the original network (taken with mass-action kinetics).  We do not consider translated networks in this work.
\end{rmk}

\begin{table}[ht]
\centering
\begin{tabular}{|l | l | l | l | }
\hline
	& & Steady-state  \\
                                                & Identifiability & invariants    \\
\hline
Glue over complexes                             &                 &  Theorem~\ref{thm:mono-single-species}                               \\
&                   &   Theorem~\ref{thm:mono-glue-species-invt}            \\
\hline
Glue over reactions                             &                &                        Theorem~\ref{thm:mono-reaction}                      \\
\hline
Join by a new reaction     &  Theorem~\ref{thm:super-i-o}             &               \\
& Theorem~\ref{thm:super-i-o-iff} &                \\
\hline
Join by replacing reactions                          &  Theorem~\ref{thm:join-new-rxn-iden}              &                   \\
& Theorem~\ref{thm:sub-super-i-o} &     \\
& Theorem~\ref{thm:onewayflow} & \\
\hline
\end{tabular}
\caption{Summary of results on joining networks}
\label{table:summaryjoining}
\end{table}

\begin{table}[ht]
\centering
\begin{tabular}{|l | l | l | l | }
\hline
	& & Steady-state \\
                                                & Identifiability & invariants   \\
\hline
Unglue over complexes                             &                 &  Theorem~\ref{thm:mono-single-species}                             \\
                            &                 &  Theorem~\ref{thm:mono-glue-species-invt}                                  \\
\hline
Unglue over reactions                             &                &  Theorem~\ref{thm:mono-reaction}                      \\
\hline
Decompose via a lost reaction                          &  Theorem~\ref{thm:super-i-o-iff}           
&                                            \\
\hline
Decompose by replacing reactions                          &  Theorem~\ref{thm:onewayflow}              &                                            \\
\hline
\end{tabular}
\caption{Summary of results on decomposing networks}
\label{table:summarydecomposing}
\end{table}

Our results on joining and ``decomposing'' networks are summarized in 
Tables~\ref{table:summaryjoining} and~\ref{table:summarydecomposing}.  
Additionally, examples pertaining to multistationarity and gluing over complexes or joining by a new reaction are given in Sections~\ref{sec:mss-glue} and~\ref{sec:mss-add-rxn}, 
respectively.
Some of our results on identifiability are in the context of
monomolecular networks, which can be viewed as ``linear compartmental
models'' 
 (after some ``input'' and ``output'' species
  are specified).  
We turn to this topic now.

\subsection{Monomolecular networks and linear compartmental models} \label{subsec:monomolecular}
A special class of reaction networks that we will consider is that of monomolecular networks.  
Recall that this means that each complex of the network is either a single species (e.g., $X_1$ or $X_2$) or the zero complex.
The associated differential equations~\eqref{eq:ODE-mass-action} therefore are linear; the general form is:
\begin{equation} \label{eq:monomolecular}
\frac{dx(t)}{dt}=A~x(t)+ u ~,
\end{equation}
where $A$ is a matrix with nonnegative off-diagonal entries, and $u$
is a nonnegative vector of inflow rates.  Both $A$ and $u$ are
composed of rate-constant parameters (and some zeroes).

Monomolecular networks have many applications in areas such as
pharmacokinetics, cell biology, and ecology, and
they commonly arise as part of 
\textit{linear compartmental models}~\citep{godfrey}. 
In this setting, the {\em input} vector $u$ 
is viewed as a control vector $u(t)$ (at least one component of $u$ is
assumed to be controlled, 
which is unlike in standard mass-action kinetics,
 and the non-controllable components $u_i(t)$ are constants). 
Thus, equation~\eqref{eq:monomolecular} becomes\footnote{The standard definition of a linear compartmental model incorporates an extra matrix $B$ as follows: 
$\frac{dx(t)}{dt}=A~x(t)+B~u(t) $;
our work therefore considers, for simplicity, the case
when $B$ is the identity matrix.  We hope in the future to extend our
results to accommodate more general $B$.}:
\begin{equation} \label{eq:linear}
	\frac{dx(t)}{dt}=A~x(t)+u(t) ~,
\end{equation}
and the matrix $A$ is called the {\em compartmental matrix}. 
Also, each species concentration $x_i(t)$ is called a {\em state variable} in this setting, representing the concentration of material in {\em compartment} $i$.  Note that $u_{i}(t) \equiv 0$ when there is no inflow of material to compartment $i$ (i.e., no {\em inflow} reaction $0 \to X_i$).  
{\em Outflow} reactions of the form $X_i \to 0$ are called {\em leaks}. 
The dictionary between these terms is in Table~\ref{table:dictionary}.

\begin{table}[ht]
\centering
\begin{tabular}{ll}
\hline
{\bf Reaction networks}     	& {\bf Compartmental models}       \\  \hline
Monomolecular network 	& Linear compartmental model \\
Species                      		&            Compartment              \\
Species concentration 
				                &  State variable                         \\
Inflow reaction 
			 (production)   &         Input                \\
Outflow reaction 
			 (degradation) &      Leak                    \\ \hline
\end{tabular}
\caption{Dictionary between reaction networks and compartmental models.
\label{table:dictionary} }
\end{table}

For identifiability problems, we assume as part of the setup that some of the species concentrations $x_i(t)$ can be observed.
 This is summarized as an {\em output} (or measurement) vector $z(t)$, in which each coordinate\footnote{The standard definition of a linear compartmental model incorporates an extra matrix $C$ as follows: 
$z(t)=C~x(t) $;
our work therefore considers the case when each row of
  $C$ is a canonical-basis vector.}
is one of the observed species concentrations $x_i(t)$. In 
literature, the vector $y(t)$ is usually used, but we use $z(t)$ to reserve $y$ for complexes.

Alternatively, we can define a linear compartmental model in terms of a directed graph 
$\mathfrak{G} = (V,E)$ with vertex set $V$ and set of directed edges $E$,
and three sets $In, Out,$ $Leak \subseteq V$.
Each vertex $i \in V$ is a
compartment in the model, while each edge $j \rightarrow i$ in $E$ represents  
the flow of material (reaction) from the $j$-th 
 to the
$i$-th compartment. The sets 
$In, Out, Leak$ are the
sets of input (inflow-reaction), output, and leak (outflow-reaction) compartments, 
respectively. Thus, we can write a linear compartmental model $\mathcal{M}$ as $\mathcal{M}=(\mathfrak{G}, In, Out, Leak)$.    

\begin{rmk} We use the convention in this paper that, for linear compartmental models, the rate constant describing the reaction from the $j$-th compartment to the $i$-th compartment is written as $a_{ij}$, whereas for monomolecular networks (and for chemical reaction networks, in general) we use $\kappa_{ji}$ to describe the reaction rate constant from species $X_j$ to species $X_i$.  
\end{rmk}

\begin{eg} The chemical reaction network $\left\{ 0 \xrightarrow{u_1} X_1 \overset{\kappa_{21}}{\underset{\kappa_{12}}\leftrightarrows} X_2 \xrightarrow{\kappa_{20}} 0 \right\}$ is a 
monomolecular network with ODEs as follows (when we view
  the inflow rate $u_1$ as time dependent):

\begin{equation*}
\begin{pmatrix}
x_1' \\ 
x_2' \end{pmatrix} =  {\begin{pmatrix} 
-\kappa_{12} & \kappa_{21} \\
\kappa_{12} & -\kappa_{20}-\kappa_{21} 
\end{pmatrix}}{\begin{pmatrix}
x_1 \\
x_2 \end{pmatrix} } + {\begin{pmatrix}
u_1(t) \\
0 \end{pmatrix}}~.
\end{equation*}
\noindent
If we view the network as a linear compartmental model, we 
use the following notation:

\begin{equation*}
\begin{pmatrix}
x_1' \\ 
x_2' \end{pmatrix} =  {\begin{pmatrix} 
-a_{21} & a_{12} \\
a_{21} & -a_{02}-a_{12} 
\end{pmatrix}}{\begin{pmatrix}
x_1 \\
x_2 \end{pmatrix} } + {\begin{pmatrix}
u_1(t) \\
0 \end{pmatrix}}~.
\end{equation*}

 If we assume a measurement (output) from the first compartment, we have an additional equation $z_1(t) = x_1(t)$, which we call an \emph{output equation}.
\end{eg}


\section{Identifiability} \label{sec:iden}
We are interested in two identifiability problems for linear and nonlinear state space models.  The first concerns joining two identifiable submodels.  The second concerns restricting a model to smaller components (subnetworks).  

\subsection{Background: identifiability and input-output equations} \label{sec:bkrd-i-o}

Structural identifiability, which was introduced by \citet{Bellman}, concerns whether it is possible to uniquely recover the parameter values of a model given perfect input-output data.  Numerous techniques to address this question have been developed \citep{Chappell, DenisVidal, Evans, global-id, sontag-dynamic}, and a particularly fruitful approach involves using differential algebra.
This approach, which was introduced by \citet{Ljung} and \citet{Ollivier}, is described briefly below.

The setup for an identifiability problem is as follows.  
A {\em model} consists of the following:
\begin{enumerate}[label=(\roman*)]
\item parametrized differential equations  -- in our setting, mass-action differential equations~\eqref{eq:ODE-mass-action} arising from a network $G$ where the parameters are the rate constants, and
\item a specification of which compartments (e.g., species) have inflow rates that are controlled by the experimenter (these rates $u_i(t)$ are called {\em input} variables) and which are 
{\em output} variables (there must be at least one output variable). The reactions associated to the inflows are incorporated in the differential equations, while the specification of output variables yields additional equations called the \textit{output equations}.
\end{enumerate}
We 
assume that the resulting output vector $z(t)$ can be measured.  
That is, we assume perfect (noiseless) input-output data $(u(t), z(t))$.

The first step of the differential algebra approach 
transforms the state space equations (that is, the differential equations of the model in which $u(t)$ is the vector of inflow-rate constants for all input vectors) into a system of differential equations, called {input-output equations}, that involve only the parameters, input variables, output variables, and their derivatives.  More precisely, the parametrized differential equations, the output equations, and each of their $M$ derivatives (where $M$ is the number of output variables) generate an ideal, and then, using Gr\"obner bases, all species concentrations (equivalently, state variables) except the input and output variables are eliminated (equivalently, the ideal is intersected with the subring with only input and output variables and their derivatives) \citep{meshkat-rosen-sullivant}.  

Equations in this elimination ideal, the {\em input-output equations}, involve only the parameters, input variables, output variables, and their derivatives. 
Each input-output equation therefore has the following form:
\begin{align} \label{eq:i-o}
\sum_i {c_i( \kappa ) ~ \psi_i (u,z) =0}~,
\end{align}
where the sum is finite,  
the coefficients $c_i( \kappa )$ are rational functions in the parameter vector $\kappa =(\kappa_1, \dots, \kappa_r)$, and 
the $\psi_i (u,z)$'s are differential monomials in $u(t)$ and $z(t)$. 

Another method for finding input-output equations is to form the \textit{characteristic set}, defined precisely by \citet{saccomani2003parameter}.  This is a triangular system that generates
the same dynamics as the original system.  
The equations in this triangular system that involve only the input variables, output variables, and parameters, 
generate the input-output equations.  Also, if the derivatives of the state variables do not appear in the last $n$ equations of the characteristic set (here $n$ is the number of state variables), the model is \textit{algebraically observable} \citep{saccomani2003parameter}, i.e., the last $n$ equations of the characteristic set involve polynomials purely in $u(t), u'(t),..., z(t), z'(t),..., \kappa$, and $x_i(t)$ for each state variable $x_i(t)$.  In this case, as stated in the literature, ``one can, in principle, solve for $x_1,...,x_n$ in the triangular set of algebraic equations recovering the state as an (instantaneous) function of the input-output variables and their derivatives'' \citep{saccomani2003parameter}.  One can also define algebraic observability without reference to the characteristic set \citep{DiopWang}.
 
Regardless of the method of obtaining input-output equations, we choose $M$ monic, 
algebraically independent input-output equations (where $M$ is again the number of output variables) \citep{Ollivier}.   
Assume, additionally, that each such input-output equation is {\em minimal} in the following sense: there is no nonzero input-output equation involving a strict subset of the monomials $\psi_i (u,z)$ as in~\eqref{eq:i-o}.
Now consider the vector of {\em all} of their coefficients $c=(c_1( \kappa ),...,c_T ( \kappa ))$. 
This 
induces a map $c:  \R^r  \rightarrow  \R^T$,
called the {\em coefficient map}.

The next step of the differential algebra approach assumes that the coefficients $c_i( \kappa )$ of the input-output equations
can be recovered uniquely from input-output data, and thus are 
presumed to be known quantities~\citep{SoderstromStoica}.  This assumption is reasonable because, given perfect data, we have values for $u(t), u'(t),u''(t),...$ and $z(t), z'(t), z''(t),...$ at many time instances.  This results in a system of linear equations in the coefficients $c_i( \kappa )$, and so, for a general input function $u(t)$ and generic parameters, there is a unique solution for the coefficients $c_i( \kappa )$.

Therefore, the identifiability question is: Can the parameters of the model be recovered 
from the coefficients of the input-output equations?  

\begin{defn}[Preliminary definition of identifiability] \label{defn:identify}  
Consider a model, and let $c$ denote its coefficient map. 
\begin{itemize}
	\item  The model is   \textit{generically globally identifiable} if
	there is a dense open subset $\Omega \subseteq \R^r$ such that
	$c: \Omega \rightarrow \R^T$ is one-to-one.
	\item The model is \textit{generically locally identifiable} 
	if there is a dense open subset $\Omega \subseteq \R^r$ such that
	around every $\kappa \in \Omega$ there is an open neighborhood $U_\kappa \subseteq \Omega$
	such that $c : U_\kappa  \rightarrow \R^T$  is one-to-one.  
	\item The model is \textit{generically unidentifiable} if there is a dense subset
	$\Omega \subseteq \R^r$ such that  $c^{-1}(c(\kappa))$ is infinite for all $\kappa \in \Omega$. 
\end{itemize}
\end{defn}
\noindent
This ability to distinguish 
between local and global identifiability sets the differential algebra approach apart from other 
methods to analyze identifiability, such as the similarity transformation approach \citep{Chappell, Evans},
which can detect local identifiability only. 

Identifiability 
is well defined~\citep{Ollivier}.

\begin{rmk} 
In this paper, we focus on generic identifiability, so we will say ``globally identifiable'' in place of ``generically globally identifiable''.  Similarly, ``locally identifiable'' or ``unidentifiable'' will mean generically so.  Furthermore, for brevity, we will simply say ``identifiable'' when we mean ``locally (respectively, globally) identifiable."
The locus of non-generic parameters, for linear compartmental models, was analyzed by \citet{singularlocus}.
\end{rmk}

\begin{rmk} \label{rmk:parameter-space}
In many applications, it is reasonable to restrict the domain of the coefficient map $c$ to
some natural, open, biologically relevant parameter space $\Theta \subseteq \R^r$.
For instance, $\Theta= \R^r_{>0}$ is an appropriate parameter space for the vector of rate constants $\kappa$.
Here, however, we use $\R^r$ to be consistent with the literature on compartmental models.
\end{rmk}

In several results (Theorems~\ref{thm:join-new-rxn-iden}
  and~\ref{thm:sub-super-i-o} and Corollary~\ref{cor:join-several}), 
we will use a notion of identifiability that generalizes Definition~\ref{defn:identify} in two ways.  
We now explain the motivation behind these two generalizations.  
First, we wish to allow for identifiability under ``changes of variables'' as follows.  Consider two models $\mathcal{M}$ and $\mathcal{M}'$, where $\mathcal{M}'$ is identifiable. 
Assume also that starting from the ODEs of $M$, after replacing input variables $u_i$ of $\mathcal{M}$ with some known functions $\hat u_i$ of measurable quantities (e.g., output variables), we obtain precisely the ODEs of $\mathcal{M}'$.  
Then, if we have input-output data $(u(t), u'(t), \dots , z(t), z'(t), \dots)$ at many time points for $\mathcal{M}$, we can compute 
$(\hat u(t), \hat u'(t), \dots)$, and then use this as part of the input-output data for $\mathcal{M}'$, thereby recovering the parameters.  
It is therefore reasonable to say that $\mathcal{M}$ is identifiable.  Such an argument was used, for instance, in the proof of Proposition 6 in the article of 
~\citet{MeshkatSullivantEisenberg}.

Secondly, we will extend the definition of identifiability to allow for adding inputs.  
The motivation is as follows.  Suppose a model $\mathcal{M}$ is obtained from a model $\mathcal{N}$ by adding one or more inputs.  Then an experimenter could collect data from $\mathcal{M}$ {\em without} using the extra inputs, so these data would effectively arise from model $\mathcal{N}$.  So, if $\mathcal{N}$ is identifiable, we also want to say that $\mathcal{M}$ is identifiable.  

Accordingly, we allow both types of extension in the following recursive definition.

\begin{defn}
\label{def:extended-iden}
A model $\mathcal{M}$ is {\em locally (respectively, globally) identifiable} if
$\mathcal{M}$ is  locally (respectively, globally) identifiable as in Definition~\ref{defn:identify} or if 
 there exist: 
\begin{enumerate}
	\item a subset $\{ \lambda_1, \dots, \lambda_k \}$ of the set of parameters $\{\kappa_1,\dots, \kappa_r\}$ of $\mathcal{M}$ 
(as shorthand, we write $\kappa=(\lambda,\mu) \in \mathbb{R}^k \times \mathbb{R}^{r-k}$),
	\item a dense open subset $\Omega \subseteq \mathbb{R}^r$, 
such that for all $\kappa^*=(\lambda^*,\mu^*) \in \Omega$, there exist only finitely many (respectively, a unique) $\lambda^{**} \in \mathbb{R}^k$ such that 
\[ 
c(\lambda^{**},\mu^*)~=~ c(\lambda^*,\mu^*)~,
\]
where $c:\mathbb{R}^r \to \mathbb{R}^T$ is the coefficient map of $\mathcal{M}$,
	\item nested subsets 
		$\{x_{i_1},\dots, x_{i_k}\} \subseteq \{ x_{j_1},\dots, x_{j_{\ell}} \} $
	of the state variables $\{x_1,\dots, x_n\}$ of $\mathcal{M}$, such that $x_{i_1},\dots, x_{i_k}$ are {\em not} input variables of $\mathcal{M}$,
	\item an $\mathbb{R}^{\ell}$-valued function $g(\gamma,\widetilde u;~ x_{j_1},\dots, x_{j_{\ell}}) $ that depends on 
	(a) a vector $\gamma$ of some parameters of $\mathcal{M}$ that are disjoint from $\lambda$,
	(b) a vector $\widetilde u$ of some of the inputs of $\mathcal{M}$, and 
	(c) the variables $ x_{j_1},\dots, x_{j_{\ell}}$,
	\item a non-constant function $f_i$ (for every $i=1, \dots,
          \ell $) 
	of 
the input and output variables of $\mathcal{M}$, their derivatives, and also the $\lambda_i$'s,
\end{enumerate}
such that the following hold:
\begin{enumerate}[label=(\roman*)]
\item the ODEs of $\mathcal{M}$ for the state variables $x_{j_1},\dots, x_{j_{\ell}}$ are as follows:
\begin{align} \label{eq:ext-iden}
	\begin{pmatrix}
	 x'_{j_1} \\ 
	\vdots \\
	x'_{j_{\ell}}
	\end{pmatrix}
~& = ~
	g(\gamma, \widetilde u ;~x_{j_1},\dots, x_{j_{\ell}}) + 
	\left(f_1{\bf e}_{i_1} + \dots + f_{\ell}{\bf e}_{i_{k}} \right)~,
\end{align}
	where ${\bf e}_i$ denotes the $i$-th canonical basis vector in
        $\mathbb{R}^{ \ell }$, 
\item 
	when each $f_q$ in the equations~\eqref{eq:ext-iden} is replaced by a new variable $\hat u_{i_q}$, then the resulting ODEs are those of a model $\mathcal{M}'$ (with state variables 
	$x_{j_1},\dots, x_{j_{\ell}}$, parameters $\gamma$, and inputs $\widetilde u$ and $\hat u$), and 
\item when $\mathcal{M}'$ is taken so that the output variables are precisely those of $\mathcal{M}$ in 
$\{ x_{j_1},\dots, x_{j_{\ell}} \}$, then $\mathcal{M}'$ is locally (respectively, globally) 
 identifiable or can be obtained from some locally (respectively, globally) identifiable model by adding one or more inputs.
\end{enumerate}
\end{defn}

We do not know whether Definition~\ref{def:extended-iden} encompasses more models than Definition~\ref{defn:identify}, so we pose the question here.

\begin{question} \label{q:iden}
Is there a model that is identifiable in the sense of Definition~\ref{def:extended-iden}, but not
in the sense of Definition~\ref{defn:identify}? 
\end{question}

The differential algebra approach to identifiability has been used to analyze models in systems biology, 
e.g., via the software {\tt DAISY} by \citet{DAISY} (see also software comparisons by \citet{global-id}), but has received surprisingly little attention in  the reaction network community.
That is not to say that few identifiability analyses have been performed on reaction networks, only that such investigations used other techniques~\citep{Chis, Davidescu}, focused on somewhat different questions, or both~\citep{GHRS}.  One such work is that of Craciun and Pantea, which we describe now.

Craciun and Pantea answered the following questions: 
when can the rate constants of a reaction network be recovered given its dynamics, and also when can the reaction network itself (the set of reactions, but not their rate constants) be recovered from its dynamics~\citep{CraciunPantea}? For the former question, the ``dynamics'' refers to time-course data $x(t)$ (all variables are therefore viewed as output, i.e., measurable, variables).
  This is a natural starting point when considering identifiability problems arising from reaction networks. Also, their results yield sufficient conditions for a network to be unidentifiable (in the sense of Definition~\ref{defn:identify}), i.e. if the network is unidentifiable with all state variables measured, then the network is unidentifiable 
  when only a subset of state variables are measured.  These results, to our knowledge, are the only general results pertaining to identifiability of reaction networks.   

In this section, we prove more results that apply to general networks. 
Note, however, that our setup differs from that by~\citet{CraciunPantea}:
we assume the network is known, but that only some of the concentrations $x_i(t)$ can be measured, and then aim to recover the rate constants.

More precisely, we focus on models $(G, \II, \OO)$ defined by a reaction network $G=(\SS,\CC,\RR)$, input set $\II\subseteq \SS$, and output set $\OO\subseteq \SS$.   
Also, we make the following assumption:
\begin{center}
{\em the set of input species consists of all inflow-reaction species, i.e.: $\mathcal{I} ~=~ \{X_i \mid  0 \to X_i {\rm~is~a~reaction~in~}G \}$}.
\end{center}
A model therefore is specified by a network $G$ and its output-species set $\OO$, and so we will write $(G,\OO)$ in place of $(G, \II, \OO)$.

\begin{notation}  \label{notation:model}
Following the literature, we indicate output species, when depicting reaction networks, by this symbol: \begin{tikzpicture}[scale=0.7]
 	\draw (4.66,-.49) circle (0.05);	
	 \draw[-] (5, -.15) -- (4.7, -.45);	
\end{tikzpicture} .
For instance, the monomolecular network 
depicted below, which arises from the network $G= \{0 \to X_1 \rla X_2 \to 0\}$, has one input species ($\II = \{X_1\}$) and one output species ($\OO= \{ X_2 \}$):
\begin{center}
	\begin{tikzpicture}[scale=1.5]
    	\node[] at (0, 0) {0};
    	\node[] at (1, 0) {$X_1$};
    	\node[] at (2, 0) {$X_2$};
    	\node[] at (3, 0) {0};
	 \draw[->] (0.2, 0) -- (0.8, 0);	
	 \draw[->] (1.2, 0.05) -- (1.8, 0.05);	
	 \draw[<-] (1.2, -0.05) -- (1.8, -0.05);	
	 \draw[->] (2.2, 0) -- (2.8, 0);	
 	\draw (1.66,-.49) circle (0.05);	
	 \draw[-] (2, -.15) -- (1.7, -.45);	
 	\end{tikzpicture}
	\end{center}
Thus, the inflow rate of the reaction $0 \to X_1$, denoted by $u_1(t)$, is assumed to be controllable, whereas the other three reaction rates are fixed constants: 
 
\begin{center}
	\begin{tikzpicture}[scale=1.5]
    	\node[] at (0, 0) {0};
    	\node[] at (1, 0) {$X_1$};
    	\node[] at (2, 0) {$X_2$};
    	\node[] at (3, 0) {0};
	 \draw[->] (0.2, 0) -- (0.8, 0);	
	 \draw[->] (1.2, 0.05) -- (1.8, 0.05);	
	 \draw[<-] (1.2, -0.05) -- (1.8, -0.05);	
	 \draw[->] (2.2, 0) -- (2.8, 0);	
    	\node[above] at (0.5, 0) {$u_1(t)$};
    	\node[above] at (1.5, 0) {$\kappa_{12}$};
    	\node[below] at (1.5, 0) {$\kappa_{21}$};
    	\node[above] at (2.5, 0) {$\kappa_{20}$};
 	\end{tikzpicture}
	\end{center}
\end{notation}

\begin{rmk}
In contrast with the general setup for identifiability analysis, the leaks in our setting are specified by the network $G$ itself, and thus need not be specified separately.
\end{rmk}

\subsection{Prior results} \label{sec:general-results}
This subsection compiles two 
results, from our work~\citep{submodel}, on identifiability of monomolecular reaction networks (i.e., linear compartmental models).  We will use these results to prove 
results on joining networks. 
For more results on identifiability of linear compartmental models, we refer the reader to~\citep{godfrey,submodel,MeshkatSullivantEisenberg}.

Proposition~\ref{prop:ioscc}, 
which is~\cite[Theorem 3.8]{submodel}, 
 states that 
 an input-output equation involving an output variable 
$z_i$ corresponds to an input-output equation 
arising from the ``output-reachable subgraph'' to $z_i$. 

\begin{defn} \label{defn:outputreachable}
For a linear compartmental model $\mathcal{M}=(\mathfrak{G}, In, Out, Leak)$, let $i \in Out$.
The {\em output-reachable subgraph to} $i$ (or {\em to} $z_i$) 
is
the induced subgraph of $\mathfrak{G}$ 
containing all vertices $j$ for which there is a directed path in $\mathfrak{G}$ from $j$ to $i$.
\end{defn}

\begin{defn}  \label{def:restrict} 
For a linear compartmental model $\mathcal{M}=(\mathfrak{G}, In, Out, Leak)$, 
let $H=(V_H,E_H)$ be an induced subgraph of $\mathfrak G$ that contains at least one output.
The {\em restriction} of $\mathcal M$ to $H$, denoted by $\mathcal{M}_H$,
is obtained from $\mathcal M$ by removing all incoming edges to $\mathfrak{G}$, retaining all leaks and outgoing edges (which become leaks), and retaining all inputs and outputs in $\mathfrak{G}$; that is,
	\[
	\mathcal{M}_H ~:=~
	(H,~ In_H,~Out_H,~Leak_H)~,
	\]  
where 
$In_H:=In \cap V_H$ and 
$Out_H:=Out \cap V_H$, and the leak set is
\[
Leak_H~:=~ \left( Leak \cap V_H \right) \cup 
	\{
	i \in V_H \mid (i,j) \in E(\mathfrak{G}) ~{\rm for~some}~ j \notin V_H
	\}~.
\]
Also, the labels of edges in $H$ are inherited from those of $\mathfrak{G}$, and labels of leaks are:
\[
	{\rm  label~of~leak~from~}k^{\rm th} {\rm~compartment}~=~ 
	\begin{cases}
		a_{0k} + \sum_{ \{ j \notin V_H \mid (k,j) \in E(\mathfrak{G}) \}} a_{jk} 
			 & {\rm if~ }k \in Leak \cap V_H\\
		 \sum_{ \{ j \notin V_H \mid (k,j) \in E(\mathfrak{G}) \}} a_{jk} 
			 & {\rm if~ }k \notin Leak \cap V_H~.
	 	\end{cases}
\]
\end{defn}

\begin{eg} \label{ex:model-restriction}
Consider the following model $\mathcal M$:
\begin{center}
	\begin{tikzpicture}[scale=1.5]
    	\node[] at (0, 0) {0};
    	\node[] at (1, 0) {$X_1$};
    	\node[] at (2, 0) {$X_2$};
    	\node[] at (3, 0) {$X_3$};
    	\node[] at (4, 0) {$X_4$};
	\node[] at (5, 0) {0};
	 \draw[->] (0.2, 0) -- (0.8, 0);	
	 \draw[->] (1.2, 0.05) -- (1.8, 0.05);	
	 \draw[<-] (1.2, -0.05) -- (1.8, -0.05);	
	 \draw[->] (2.2, 0) -- (2.8, 0);	
	 \draw[->] (3.2, 0.05) -- (3.8, 0.05);	
	 \draw[<-] (3.2, -0.05) -- (3.8, -0.05);	
	 \draw[<-] (4.2, 0) -- (4.8, 0);	
 	\draw (0.66,-.49) circle (0.05);	
	 \draw[-] (1, -.15) -- (0.7, -.45);	
 	\draw (3.66,-.49) circle (0.05);	
	 \draw[-] (4, -.15) -- (3.7, -.45);	
    	\node[above] at (0.5, 0) {$u_1(t)$};
    	\node[above] at (1.5, 0) {$a_{21}$};
    	\node[below] at (1.5, 0) {$a_{12}$};
    	\node[above] at (2.5, 0) {$a_{32}$};
    	\node[above] at (3.5, 0) {$a_{43}$};
    	\node[below] at (3.5, 0) {$a_{34}$};
    	\node[above] at (4.5, 0) {$u_4(t)$};
	\end{tikzpicture}
\end{center}
The output-reachable subgraph to $i=1$ is $X_1 \leftrightarrows X_2$.  Thus, the restriction $\mathcal{M}_H$ is as follows:
\begin{center}
	\begin{tikzpicture}[scale=1.5]
    	\node[] at (0, 0) {0};
    	\node[] at (1, 0) {$X_1$};
    	\node[] at (2, 0) {$X_2$};
    	\node[] at (3, 0) {$0$};
	 \draw[->] (0.2, 0) -- (0.8, 0);	
	 \draw[->] (1.2, 0.05) -- (1.8, 0.05);	
	 \draw[<-] (1.2, -0.05) -- (1.8, -0.05);	
	 \draw[->] (2.2, 0) -- (2.8, 0);	
 	\draw (0.66,-.49) circle (0.05);	
	 \draw[-] (1, -.15) -- (0.7, -.45);	
    	\node[above] at (0.5, 0) {$u_1(t)$};
    	\node[above] at (1.5, 0) {$a_{21}$};
    	\node[below] at (1.5, 0) {$a_{12}$};
    	\node[above] at (2.5, 0) {$a_{32}$};
	\end{tikzpicture}
\end{center}
The corresponding compartmental matrix is 
\[
	A_H ~=~ 
		{\begin{pmatrix} 
		-a_{21} & a_{12} \\
		a_{21} & -a_{12}-a_{32} 
		\end{pmatrix}} ~.
\]
\end{eg}

\begin{proposition}[Input-output equations~\citep{submodel}] \label{prop:ioscc} 
Let $\mathcal{M}=(\mathfrak{G}, In, Out, Leak)$
be a linear compartmental model. 
Let $i \in Out$, 
and assume that there exists a directed path in $\mathfrak{G}$ from some input compartment 
to compartment-$i$. 
Let $ H=(V_H,E_H)$ denote the output-reachable subgraph to~$z_i$, 
and let $A_H$ denote the compartmental matrix for the restriction $\mathcal{M}_H$. Assume $In \cap V_H$ is nonempty. 
Define $\partial I$ to be the $|V_H| \times |V_H|$ matrix in which every diagonal entry is the differential operator $d/dt$ and every off-diagonal entry is 0.  
Then the following is an input-output equation for $\mathcal M$: 
 \begin{align}  \label{eq:i-o-for-general-model}
 	\det (\partial I -{A}_H) z_i ~=~   \sum_{j \in In \cap V_H} (-1)^{i+j}
		 \det \left( \partial I-{A}_H \right)_{ji} u_j ~,
 \end{align}
where $ \left( \partial I-{A}_H \right)_{ji}$ denotes the matrix obtained from
 $\left( \partial I-{A}_H \right)$ by removing the row corresponding to compartment-$j$ and the column corresponding to compartment-$i$.
Thus, this input-output equation~\eqref{eq:i-o-for-general-model}
involves only the output-reachable subgraph to $z_i$.

\end{proposition}

\begin{eg}[Example~\ref{ex:model-restriction}, continued] \label{ex:model-restriction-2}
We continue with the model $\mathcal{M}$ and the restriction $\mathcal{M}_H$ (with compartmental matrix $A_H$) displayed earlier in Example~\ref{ex:model-restriction} .  By Proposition~\ref{prop:ioscc}, an input-output equation for $\mathcal M$ involving output variable $z_1$ is as follows:
\[
	\det (\partial I -{A}_H) z_1 ~=~    (-1)^{1+1}
		 \det \left( \partial I-{A}_H \right)_{11} u_1 ~,
\]
which reduces to
\[
	z_1^{(2)} + (a_{12}+a_{21}+a_{32})z_1' + a_{21}a_{32} z_1 ~=~ u_1' + (a_{12}+a_{32})u_1~.
\]
\end{eg}

\begin{rmk} \label{rmk:i-o}
In Section~\ref{sec:monomol-iden}, 
we will analyze identifiability using the coefficient maps arising from the input-output equations~\eqref{eq:i-o-for-general-model}.
\end{rmk}

The next result, which is~\cite[Theorem 4.3]{submodel}, 
analyzes the effect of adding an outflow. 

\begin{definition} \label{def:non-flow}
The {\em non-flow subnetwork} of a reaction network $G$
is the subnetwork obtained by removing from $G$ the zero complex, all outflow reactions (leaks), and inflows. 
\end{definition}

\begin{lemma}[Adding one outflow \citep{submodel}]
\label{lem:add-1-leak}
Let $G=(\mathcal S, \mathcal C, \mathcal R)$ be a monomolecular reaction network with no outflow reactions 
and at least one inflow reaction.
Assume that the non-flow subnetwork of $G$ is strongly connected. 
 Let $\mathcal O \subseteq \mathcal S$, and 
  let $\widetilde G$ be obtained from $G$ by adding one outflow reaction. 
Then, if $(G, \mathcal O)$ is generically locally identifiable, 
then so is $(\widetilde G, \mathcal O)$.  
\end{lemma}

\subsection{Joining by replacing reactions}

This section considers the question, 
{\em After joining two identifiable networks by replacing reactions, is the resulting network identifiable?}  
Theorem~\ref{thm:join-new-rxn-iden} states that the answer is `yes' if 
the two networks are joined by a ``one-way flow'' (see Definition~\ref{def:one-way-flow}), 
the two networks have disjoint sets of species, and
the first network is algebraically observable.

Models joined by a ``one-way flow'' are considered by \citet{MeshkatSullivantEisenberg} and are common in physiologically based pharmacokinetic models (see e.g. \citep{DiStefano1988, McMullin2003, Pilo1990}), where often one models the pharmacokinetics of a substance and its metabolites (so that each step in the metabolism of the substance forms a `tier' in the overall model). These structures are also common in aging models, wherein individual movement or states are modeled as a single submodel, and then a discrete aging process is included, generating multiple copies of the submodel connected by a one-way flow \citep{MeshkatSullivantEisenberg}.

Let us precisely explain what we mean by a ``one-way flow''.  There are four
scenarios considered in this section.  In the first, one or more outflow reactions (leaks) $X_i \to 0$ in one network correspond to some $0 \to X_j$'s in the other network, i.e. each leak in the first network is an input in the second.  Joining these networks therefore creates new reactions $X_i \to X_j$, as summarized here:

\begin{center}
\begin{tikzpicture}[scale=1]
	\node[left] at (-0.5,0) {{\sc Scenario 1:} \quad Joining};
	\draw [black] (-.3, -1) rectangle (2, .5);
    	\node[right] at (0.3, 0) {$X_i \to 0$};
	\node[right] at (1.0,-0.3) {$\vdots$};
    	\node[right] at (2.2, 0) {and};
	\draw [black] (3.2, -1) rectangle (5.8, .5);
    	\node[left] at (5, 0) {$0 \to X_j$};
	\node[right] at (3.8,-0.3) {$\vdots$};
    	\node[right] at (6.3, 0) {yields};
	\draw [black] (8, -1) rectangle (11.1, .5);
    	\node[right] at (8.5, 0) {$X_i \to X_j$};
	\node[right] at (9.2,-0.3) {$\vdots$};
\end{tikzpicture}
\end{center}
In the second scenario, certain reactions $X_i \to 0$ are replaced by new reactions $X_i \to X_j$:
\begin{center}
\begin{tikzpicture}[scale=1]
	\node[left] at (-0.5,0) {{\sc Scenario 2:} \quad Joining};
	\draw [black] (-.3, -1) rectangle (2, .5);
    	\node[right] at (0.3, 0) {$X_i \to 0$};
	\node[right] at (1.0,-0.3) {$\vdots$};
    	\node[right] at (2.2, 0) {and};
	\draw [black] (3.2, -1) rectangle (5.8, .5);
    	\node[left] at (5, 0) {$X_j$};
	\node[right] at (4.3,-0.3) {$\vdots$};
    	\node[right] at (6.3, 0) {yields};
	\draw [black] (8, -1) rectangle (11.1, .5);
    	\node[right] at (8.5, 0) {$X_i \to X_j$};
	\node[right] at (9.2,-0.3) {$\vdots$};
\end{tikzpicture}
\end{center}
In the third scenario, the new reactions $X_i \to X_j$ are added, and none are replaced:
\begin{center}
\begin{tikzpicture}[scale=1]
	\node[left] at (-0.5,0) {{\sc Scenario 3:} \quad Joining};
	\draw [black] (-.3, -1) rectangle (2, .5);
    	\node[right] at (0.3, 0) {$X_i$};
	\node[right] at (0.4,-0.3) {$\vdots$};
    	\node[right] at (2.2, 0) {and};
	\draw [black] (3.2, -1) rectangle (5.8, .5);
    	\node[left] at (5, 0) {$X_j$};
	\node[right] at (4.3,-0.3) {$\vdots$};
    	\node[right] at (6.3, 0) {yields};
	\draw [black] (8, -1) rectangle (11.1, .5);
    	\node[right] at (8.5, 0) {$X_i \to X_j$};
	\node[right] at (9.2,-0.3) {$\vdots$};
\end{tikzpicture}
\end{center}
In the fourth scenario, certain reactions $0 \to X_j$ are replaced by new reactions $X_i \to X_j$:
\begin{center}
\begin{tikzpicture}[scale=1]
	\node[left] at (-0.5,0) {{\sc Scenario 4:} \quad Joining};
	\draw [black] (-.3, -1) rectangle (2, .5);
    	\node[right] at (0.3, 0) {$X_i$};
	\node[right] at (0.4,-0.3) {$\vdots$};
    	\node[right] at (2.2, 0) {and};
	\draw [black] (3.2, -1) rectangle (5.8, .5);
    	\node[left] at (5, 0) {$0 \to X_j$};
	\node[right] at (3.8,-0.3) {$\vdots$};
    	\node[right] at (6.3, 0) {yields};
	\draw [black] (8, -1) rectangle (11.1, .5);
    	\node[right] at (8.5, 0) {$X_i \to X_j$};
	\node[right] at (9.2,-0.3) {$\vdots$};
\end{tikzpicture}
\end{center}

Here we define these scenarios precisely:
\begin{defn} \label{def:one-way-flow}
Let $N_1=(\SS_1,\CC_1,\RR_1)$ and $N_2=(\SS_2,\CC_2,\RR_2)$ be reaction networks with disjoint sets of species 
$\SS_1=\{X_1,\dots, X_m\}$ and $\SS_2 = \{X_{m+1}, \dots, X_n\}$.
A network $G$ is obtained by {\em joining $N_1$ and $N_2$ by a one-way flow} 
if there exist a nonempty subset $\mathfrak{I} \subseteq [m]$ and a function $\phi: \mathfrak{I} \to \{m+1, \dots, n\}$ 
such that one of the following holds:
\begin{itemize}
  \item {\sc Scenario 1:} 
The set  $\mathcal{R}_1' := \{X_i \to 0 \mid i \in \mathfrak{I} \}$ is a set of outflow reactions of $N_1$, 
	the set $\mathcal{R}_2':=\{0 \to X_{\phi(i)} \mid i \in \mathfrak{I} \}$ is a set of inflow reactions of $N_2$, and $G$ is obtained by joining $N_1$ and $N_2$ by replacing $\mathcal{R}_1' \cup \mathcal{R}_2'$ by $\{X_i \to X_{\phi(i)} \mid i \in \mathfrak{I} \}$.
  \item {\sc Scenario 2:} The set  $\mathcal{R}_1' := \{X_i \to 0 \mid i \in \mathfrak{I} \}$ is a set of outflow reactions of $N_1$, and $G$ is obtained from $N_1$ and $N_2$ by replacing $\mathcal{R}_1'$ by $\{X_i \to X_{\phi(i)} \mid i \in \mathfrak{I} \}$.
  \item {\sc Scenario 3:} $G$ is obtained by joining $N_1$ and $N_2$ by the 
new reactions $\{X_i \to X_{\phi(i)} \mid i \in \mathfrak{I} \}$.
  \item {\sc Scenario 4:} The set $\mathcal{R}_2':=\{0 \to X_{\phi(i)} \mid i \in \mathfrak{I} \}$ is a set of inflow reactions of $N_2$, 
and $G$ is obtained from $N_1$ and $N_2$ by replacing $\mathcal{R}_2'$ by $\{X_i \to X_{\phi(i)} \mid i \in \mathfrak{I} \}$.
\end{itemize}
\end{defn}

Recall our assumption that the set of input species in a model consists of all inflow-reaction species.  Then 
this set, for the network obtained by joining by a one-way flow (Definition~\ref{def:one-way-flow}), is as follows.
 Let $\mathcal{I}_i \subseteq  \SS_i$ be the input-species set for species set $\SS_i$ for $i \in \left\{1,2\right\}$.  Let 
\begin{align*}
\mathcal{I}'_2 ~:=~ 
	\begin{cases}
      \mathcal{I}_2 - \left\{X_{\phi(i)} \mid i \in \mathfrak{I} \right\}~ & \text{if $G$ is obtained via Scenario 1 or 4} \\
      \mathcal{I}_2 ~ & \text{if $G$ is obtained via Scenario 2 or 3.}
    \end{cases}
\end{align*} 
Then the input-species set for the joined network $G$ is $\mathcal{I}_1 \cup \mathcal{I}'_2$.

Consider a network $G$ obtained by joining $N_1=(\SS_1,\CC_1,\RR_1)$ and $N_2=(\SS_2,\CC_2,\RR_2)$ 
by a one-way flow (via a joining function $\phi: \mathfrak{I} \to \{m+1, \dots, n\}$).
Let $\mathcal O_1 \subseteq \SS_1$ and 
$\mathcal O_2 \subseteq \SS_2$ be nonempty.  Then $(G, \mathcal O_1 \cup \mathcal O_2)$ is the {\em model obtained by joining 
$(N_1, \mathcal O_1)$ and 
$(N_2, \mathcal O_2)$} (via $\phi$).

Our first main result generalizes~\cite[Proposition 6]{MeshkatSullivantEisenberg}, which analyzed a subcase of Scenario~1.

\begin{theorem}  \label{thm:join-new-rxn-iden}
Let $N_1=(\SS_1,\CC_1,\RR_1)$ and $N_2=(\SS_2,\CC_2,\RR_2)$ be reaction networks with disjoint sets of species. 
Let $\mathcal{O}_1 \subseteq  \SS_1$ and 
 $\mathcal{O}_2 \subseteq  \SS_2$
 be nonempty. 
Assume $(N_1, \mathcal{O}_1)$ is algebraically observable. 
Let $G$ be a network obtained by joining $N_1$ and $N_2$ by a one-way flow via Scenario 1 or 2.  
\noindent
Then, if $(N_1,\mathcal{O}_1 )$ and $(N_2,\mathcal{O}_2)$ are identifiable, then 
$(G,\mathcal{O}_1 \cup \mathcal{O}_2)$ 
is identifiable. 
\end{theorem}

\begin{proof} 
Let $N_1$, $N_2$, and $G$ be as in the statement of the theorem. Then network $G$ arises, as in Definition~\ref{def:one-way-flow}, by way of a set $\mathfrak{I}$ and a joining function~$\phi$. 

We consider first the case of Scenario 1.
We write the 
ODEs of $N_1$ as follows:
\begin{align} \label{eq:N1-ODEs}
\begin{pmatrix}
x'_1 \\ 
\vdots \\
x'_m
\end{pmatrix}
~ = ~
f(\alpha, u^{(1)};~x_1,\dots, x_m) - \sum_{i \in \mathfrak{I}} \beta_i x_i {\bf e}_i~,
\end{align}
where $u^{(1)}=u^{(1)}(t)$ is the input vector (that is, the experimenter-controlled vector of inflow rates for the species in $\mathcal{I}_1$), 
$\alpha$ is the vector of non-inflow rate constants for reactions {\em not} in $\mathcal{R}_1' = \{X_i \to 0 \mid i \in \mathfrak{I} \}$, and 
$\beta_i$, for $i \in \mathfrak{I} $, denotes the rate constant for the 
outflow reaction $X_i \to 0$ in  $\mathcal{R}_1'$. 
Also, ${\bf e}_i$ denotes the $i$-th canonical basis vector.

Similarly, we write the ODEs of $N_2$ as follows 
(recall that we are in Scenario 1):
\begin{align}\label{eq:odenetwork2}
\begin{pmatrix}
x'_{m+1} \\ 
\vdots \\
x'_n
\end{pmatrix}
~ = ~
g(\gamma, u^{(2)} ;~ x_{m+1},\dots, x_n) + 
\sum_{j \in \phi(\mathfrak I)} \widetilde{u}^{(2)}_{0 \to X_j} {\bf e}_j~,
\end{align}
where $\gamma$ is the input vector of non-inflow rate constants, 
$\widetilde{u}^{(2)}_{0 \to X_j}=\widetilde{u}^{(2)}_{0 \to X_j}(t)$, for $j \in \phi(\mathfrak{I})$, 
is the (controlled) rate for the to-be-replaced reaction $0 \to X_j$, and $u^{(2)}=u^{(2)}(t)$ is the vector of all remaining inflow rates.

The joined network $G$ has ODEs as follows:
\begin{align} \label{eq:full-ODEs}
	\begin{pmatrix}
	x'_{1} \\ 
	\vdots \\
	x'_n
	\end{pmatrix}
		~ = ~ 
	\begin{pmatrix}
	f(\alpha, u^{(1)};~x_1,\dots, x_m)  \\ 
	g(\gamma, u^{(2)} ;~ x_{m+1},\dots, x_n)
	\end{pmatrix}
 - \sum_{i \in \mathfrak{I}} \beta_i x_i ({\bf e}_i - {\bf e}_{\phi(i)} ) ~.
\end{align}
Notice that the first $m$ of the ODEs of $G$ are equal to the 
ODEs of $N_1$, as given in~\eqref{eq:N1-ODEs}.

We claim that identifiability of $(N_1,\mathcal{O}_1)$ implies identifiability of the rate constants of the vectors $\alpha$ and $\beta$ of $G$.  To see this, we consider a coefficient map $c_{N_1}$ for $N_1$ arising from a 
choice of $|\mathcal{O}_1|$ minimal, monic, algebraically independent input-output equations of $N_1$ (which are also input-output equations of $G$), 
and then extend it to a coefficient map $(c_{N_1},\widetilde c)$ for $G$ by extending to a set of $|\mathcal{O}_1 \cup \mathcal{O}_2|$ minimal, monic, algebraically independent input-output equations of $G$.  
Thus, as $c_{N_1}$ is generically locally (respectively, globally) one-to-one, thereby allowing the vectors $\alpha$ and $\beta$ to be recovered for $N_1$, we conclude that $\alpha$ and $\beta$ can be recovered for $G$.

Thus, to finish the proof in Scenario 1, we need only show that identifiability of 
$(N_2, \mathcal{O}_2)$  
implies identifiability of the rate constants $\gamma$ for $G$.  
The last $(n-m)$ ODEs of $G$, from equation~\eqref{eq:full-ODEs}, are:
\begin{align} \label{eq:last-eqns-G}
	\begin{pmatrix}
	x'_{m+1} \\ 
	\vdots \\
	x'_n
	\end{pmatrix}
~& = ~
g(\gamma, u^{(2)} ;~x_{m+1},\dots, x_n) +  \sum_{i \in \mathfrak{I} } \beta_i x_i  {\bf e}_{\phi(i)} \\ 
\notag 
~& = ~
g(\gamma, u^{(2)} ;~x_{m+1},\dots, x_n) +  \sum_{j=m+1}^n 
		\left( \sum_{ \{ i \in \mathfrak{I} \mid \phi(i)=j \} } \beta_i x_i  \right) {\bf e}_{j}  ~.
\end{align}
As $N_1$ is algebraically observable, 
the state variables $x_1, \ldots, x_m$ can be written as a function of $u^{(1)}, z^{(1)}, \alpha,$ and $\beta$.  Therefore, for $j \in \phi(\mathfrak I)$, 
the sum 
$\sum_{ \{ i \in \mathfrak{I} \mid \phi(i)=j \} } \beta_i x_i  $
is a function of $u^{(1)}, z^{(1)}, \alpha,$ and $\beta$, and 
so we may treat these sums as 
known quantities or as controlled inflow rates, thereby recovering the parameters $\gamma$. More precisely, 
for $j \in \phi(\mathfrak I)$,
letting $\hat u_j :=
\sum_{ \{ i \in \mathfrak{I} \mid \phi(i)=j \} } \beta_i x_i  $,
then the last $(n-m)$ ODEs of $G$, in~\eqref{eq:last-eqns-G}, 
match those of the identifiable network $N_2$, in~\eqref{eq:odenetwork2}.  Hence, 
by Definition~\ref{def:extended-iden}, $G$ is identifiable.

For Scenario 2, let $N_3$ be obtained from $N_2$ by adding inflows $0 \to X_j$ (inputs) for all $j \in \phi(\mathfrak I)$.  Then, by definition, $N_3$ is identifiable, and $G$ is obtained from $N_1$ and $N_3$ by a one-way flow via Scenario 1.  So, following the above proof (for Scenario 1),
$G$ is identifiable. 
\end{proof}

We define inductively what it means to join several networks by a one-way flow.  A network is obtained by {\em joining networks $N_1, \dots, N_p$ by a one-way flow} if it results from joining, by a one-way flow,
$N_1$ and a network obtained by joining $N_2, \dots, N_{p}$ by a one-way flow. 
Similarly, a model obtained by {\em joining models $(N_1, \mathcal O_1), \dots, (N_p, \mathcal O_p)$ by a one-way flow} arises from a network obtained by joining $N_1, \dots, N_p$ by a one-way flow, and the output set is $ \mathcal O_1 \cup \dots \cup  \mathcal O_p$.

Now the following result is immediate from Theorem~\ref{thm:join-new-rxn-iden}:

\begin{corollary} \label{cor:join-several}
Let $N_1=(\SS_1,\CC_1,\RR_1), 
\dots, N_p=(\SS_p,\CC_p,\RR_p)$ be reaction networks with
pairwise disjoint sets of species. 
Let $\mathcal{O}_i \subseteq  \SS_i$ be nonempty for $i=1,...,p$. 
Assume $(N_1,\mathcal{O}_1), \dots,$ $(N_{p-1},\mathcal{O}_{p-1})$ are algebraically observable. 
Let $G$ be a network obtained by joining $N_1, 
..., N_p$ by a one-way flow via Scenario 1 or 2.  
Then, if $(N_1,\mathcal{O}_1 ),  ..., (N_p, \mathcal{O}_p)$ are identifiable, then 
$(G,\mathcal{O}_1 \cup \dots \cup \mathcal{O}_p)$ 
is identifiable.
\end{corollary}

\begin{eg} \label{eg:join-over-0-nonlinear} 
Consider three networks, which we call $(N_1, \mathcal{O}_1)$, $(N_2, \mathcal{O}_2)$, and $(N_3, \mathcal{O}_3)$:
\begin{center}
	\begin{tikzpicture}[scale=1.5]
	\draw [black] (-.1, -1.2) rectangle (4.5, .4);
    	\node[] at (1, -0.5) {0};
			\node[] at (2, -0.5) {$X_1$};
    	\node[] at (1, 0) {$2X_1+X_2$};
    	\node[] at (3, 0) {$X_1+2X_2$};
    	\node[] at (4.3, 0) {0};
   \draw[->] (1.2, -0.5) -- (1.8, -0.5);
	 \draw[->] (1.7, 0.05) -- (2.2, 0.05);	
	 \draw[<-] (1.7, -0.05) -- (2.2, -0.05);	
	 \draw[->] (3.6, 0) -- (4.1, 0);	
 	\draw (1.55,-0.95) circle (0.05);	
	 \draw[-] (1.9, -0.6) -- (1.6, -0.9);	
%
	\draw [black] (4.7, -.8) rectangle (7.2, .4);
    	\node[] at (5, 0) {0};
    	\node[] at (6, 0) {$X_3$};
    	\node[] at (7, 0) {$0$};
	 \draw[->] (5.2, 0) -- (5.8, 0);	
	 \draw[->] (6.2, 0) -- (6.8, 0);	
 	\draw (5.66,-.49) circle (0.05);	
	 \draw[-] (6, -.15) -- (5.7, -.45);	

	\draw [black] (7.7, -.8) rectangle (10.2, .4);
    	\node[] at (8, 0) {$X_3$};
    	\node[] at (9, 0) {$X_4$};
	\node[] at (10,0) {0};
	 \draw[->] (8.2, 0.05) -- (8.8, 0.05);	
	 \draw[<-] (8.2, -0.05) -- (8.8, -0.05);	
	 \draw[<-] (9.2, 0) -- (9.8, 0);		 
 	\draw (8.66,-.49) circle (0.05);	
	 \draw[-] (9, -.15) -- (8.7, -.45);	
	\end{tikzpicture}
\end{center}
Each model is globally identifiable, and $(N_1, \mathcal{O}_1)$ is algebraically observable 
(e.g., using DAISY \citep{DAISY}).  
So, by Theorem \ref{thm:join-new-rxn-iden}, the model depicted below, which is obtained by joining $N_1$ and $N_2$ via Scenario 1 (by replacing the reactions $X_1+2X_2 \to 0$ and $0 \to X_3$ by the reaction $X_1+2X_2 \to X_3$), is also globally identifiable:
\begin{center}
	\begin{tikzpicture}[scale=1.5]
    	\node[] at (1, 0) {$2X_1+X_2$};
    	\node[] at (3, 0) {$X_1+2X_2$};
    	\node[] at (4.3, 0) {$X_3$};
    	\node[] at (5.3, 0) {$0$};
			\node[] at (6, 0) {$0$};
			\node[] at (7, 0) {$X_1$};
	 \draw[->] (1.7, 0.05) -- (2.2, 0.05);	
	 \draw[<-] (1.7, -0.05) -- (2.2, -0.05);	
	 \draw[->] (3.6, 0) -- (4.1, 0);	
	 \draw[->] (4.5, 0) -- (5.1, 0);
	 \draw[->] (6.2, 0) -- (6.8, 0);
 	\draw (3.96,-.49) circle (0.05);	
	 \draw[-] (4.3, -.15) -- (4.0, -.45);	
 	\draw (6.66,-.49) circle (0.05);	
	 \draw[-] (7, -.15) -- (6.7, -.45);	
	\end{tikzpicture}
\end{center}
Similarly, by the same theorem, joining $N_1$ and $N_3$ via Scenario 2 (by replacing $X_1+2X_2 \to 0$ by $X_1+2X_2 \to X_3$), yields a model that is globally identifiable:
\begin{center}
	\begin{tikzpicture}[scale=1.5]
    	\node[] at (0, 0) {$2X_1+X_2$};
    	\node[] at (2, 0) {$X_1+2X_2$};
    	\node[] at (3, 0) {$X_3$};
    	\node[] at (4, 0) {$X_4$};
	\node[] at (5, 0) {0};
	\node[] at (6, 0) {0};
	\node[] at (7, 0) {$X_1$};
	 	
	 \draw[->] (0.7, 0.05) -- (1.2, 0.05);	
	 \draw[<-] (0.7, -0.05) -- (1.2, -0.05);	
	 \draw[->] (2.6, 0) -- (2.8, 0);	
	 \draw[->] (3.2, 0.05) -- (3.8, 0.05);	
	 \draw[<-] (3.2, -0.05) -- (3.8, -0.05);	
	 \draw[<-] (4.2, 0) -- (4.8, 0);
	  \draw[->] (6.2, 0) -- (6.8, 0);
 	\draw (6.66,-.49) circle (0.05);	
	 \draw[-] (7, -.15) -- (6.7, -.45);	
 	\draw (3.66,-.49) circle (0.05);	
	 \draw[-] (4, -.15) -- (3.7, -.45);	
	\end{tikzpicture}
\end{center}
\end{eg}

Informally, Theorem~\ref{thm:join-new-rxn-iden} above
stated the following: assuming that $(N_1, \mathcal{O}_1)$ is algebraically observable, 
if identifiable networks $N_1$ and $N_2$ are joined via Scenario 1 or 2, then the result is still identifiable. 
We now consider the converse: If the joined model is identifiable, can we conclude that  $(N_1, \mathcal{O}_1)$ and $(N_2 \mathcal{O}_2)$ are also identifiable?
For $N_2$, in general, we can not (see
Example~\ref{eg:why-can-not-conclude-N2-identifiable-nonlinear} below
and Example~\ref{eg:why-can-not-conclude-N2-identifiable} in the next
subsection); but, under extra hypotheses, we can (see Theorem~\ref{thm:onewayflow} in the next subsection).  As for $N_1$, we give a counterexample in the next subsection (see Example~\ref{eg:why-can-not-conclude-N1-identifiable}).

\begin{eg} \label{eg:why-can-not-conclude-N2-identifiable-nonlinear} 
Consider two models, which we call $(N_1, \mathcal{O}_1)$ and $(N_2, \mathcal{O}_2)$:
\begin{center}
        \begin{tikzpicture}[scale=1.5]
	\draw [black] (-.8, -1.2) rectangle (3.5, .4);
    	\node[] at (0, -0.5) {0};
			\node[] at (1, -0.5) {$X_1$};
    	\node[] at (0, 0) {$2X_1+X_2$};
    	\node[] at (2, 0) {$X_1+2X_2$};
    	\node[] at (3.3, 0) {0};
   \draw[->] (0.2, -0.5) -- (0.6, -0.5);
	 \draw[->] (0.7, 0.05) -- (1.2, 0.05);	
	 \draw[<-] (0.7, -0.05) -- (1.2, -0.05);	
	 \draw[->] (2.6, 0) -- (3.1, 0);	
 	\draw (0.55,-0.95) circle (0.05);	
	 \draw[-] (0.9, -0.6) -- (0.6, -0.9);	
%
	\draw [black] (3.7, -.7) rectangle (7.5, .4);
    	  \node[] at (4, 0) {0};
        \node[] at (5, 0) {$X_3$};
        \node[] at (6, 0) {$X_4$};
        \node[] at (7, 0) {0};
	 \draw[<-] (4.2, -0) -- (4.8, 0); 
	 \draw[->] (5.2, 0.05) -- (5.8, 0.05);	
	 \draw[<-] (5.2, -0.05) -- (5.8, -0.05);	
	 \draw[->] (6.2, 0.05) -- (6.8, 0.05);	
	 \draw[<-] (6.2, -0.05) -- (6.8, -0.05);	
 	\draw (5.66,-.49) circle (0.05);	
	 \draw[-] (6, -.15) -- (5.7, -.45);	
  \end{tikzpicture}
\end{center}
The first model is the same as in the previous example, which we noted is algebraically observable. 
The model below, obtained by joining $N_1$ and $N_2$ via Scenario 2 (by replacing $X_1+2X_2 \to 0$ with $X_1+2X_2 \to X_3$), 
is globally identifiable (e.g., using DAISY \citep{DAISY}):
\begin{center}
	\begin{tikzpicture}[scale=1.5]
    	\node[] at (-0.3, 0) {$2X_1+X_2$};
    	\node[] at (1.7, 0) {$X_1+2X_2$};
    	\node[] at (3, 0) {$X_3$};
    	\node[] at (4, 0) {$X_4$};
	\node[] at (5, 0) {0};
	\node[] at (3,-0.9) {0};
	\node[] at (6, 0) {0};
	\node[] at (7, 0) {$X_1$};
	 \draw[->] (0.4, 0.05) -- (0.9, 0.05);	
	 \draw[<-] (0.4, -0.05) -- (0.9, -0.05);	
	 \draw[->] (2.3, 0) -- (2.8, 0);	
	 \draw[->] (3.2, 0.05) -- (3.8, 0.05);	
	 \draw[<-] (3.2, -0.05) -- (3.8, -0.05);	
	 \draw[->] (4.2, 0.05) -- (4.8, 0.05);	
	 \draw[<-] (4.2, -0.05) -- (4.8, -0.05);
		\draw[->] (6.2, 0) -- (6.8, 0);
	 \draw[->] (3, -0.15) -- (3, -0.75);		 
 	\draw (3.66,-.49) circle (0.05);	
	 \draw[-] (4, -.15) -- (3.7, -.45);
 	\draw (6.66,-.49) circle (0.05);	
	 \draw[-] (7, -.15) -- (6.7, -.45);	
	\end{tikzpicture}
\end{center}
However, $(N_2, \mathcal{O}_2)$ is unidentifiable~\citep{MeshkatSullivantEisenberg}.  
\end{eg}

\subsection{Monomolecular networks} \label{sec:monomol-iden}
The previous subsection focused on networks $G$ formed by joining two networks by a one-way flow via Scenario 1 or 2.  
We examined the extent to which identifiability can be ``transferred'' from subnetworks $N_i$ to $G$ 
 (Theorem \ref{thm:join-new-rxn-iden}).

The current subsection considers the case when all networks are monomolecular (the case of linear compartmental models).  In this setting, we obtain stronger conclusions than in Theorem~\ref{thm:join-new-rxn-iden} (see Theorems \ref{thm:sub-super-i-o} and
~\ref{thm:onewayflow}).  We also consider more scenarios for joining by a one-way flow  (Theorem~\ref{thm:super-i-o} and Theorem~\ref{thm:super-i-o-iff}).  
We informally summarize our results as follows: {\em 
Let $G$ be obtained by joining 
 monomolecular networks 
$N_1$ and $N_2$ by a one-way flow via Scenario 1, 2, 3, or 4.  Then (1) if $N_1$ and $N_2$ are identifiable, then $G$ is identifiable, and (2) if $N_1$ and $G$ are identifiable in the case of Scenario 1 or 4, then $N_2$ is identifiable.}
(For the precise statements, see Theorems~\ref{thm:sub-super-i-o}, \ref{thm:onewayflow}, \ref{thm:super-i-o}, and \ref{thm:super-i-o-iff} and Corollary~\ref{cor:strongly-connected}). 

\begin{rmk}
The results in the rest of this section pertain to monomolecular networks that have at least one inflow reaction (i.e., at least one input).  This requirement allows us to 
use a prior result pertaining to input-output equations (Proposition~\ref{prop:ioscc}).
(Recall that we already required, in Section~\ref{sec:bkrd-i-o}, that every model has at least one output.)
\end{rmk}

Following \citet{submodel}, we allow identifiability of linear compartmental models to be analyzed from the input-output equations arising from output-reachable subgraphs:

\begin{center} {\em 
For monomolecular networks, we extend the definition of identifiability to allow (as in Definition~\ref{defn:identify}) coefficient maps arising from input-output equations given in~\eqref{eq:i-o-for-general-model}. }
\end{center}

It is conjectured that this extended definition is {\em not} actually an extension, i.e., that the definition does not encompass more models than the previous definition~\cite[Remark 3.10]{submodel}.

 \subsubsection{Joining output connectable, monomolecular networks via Scenario 1 or 2}
The results in the previous subsection required some of the models $(N_i,\mathcal{O}_i)$ to be 
algebraically observable. 
This condition is in general difficult to verify, but automatically holds for
monomolecular 
networks that satisfy a condition that is easier to check, namely, being ``output connectable'' 
(Definition~\ref{def:outputconnectable} and Lemma~\ref{lem:strongly-connected}).  
Therefore, we can state a version of Corollary~\ref{cor:join-several} for monomolecular networks (see Theorem~\ref{thm:sub-super-i-o}).

\begin{definition} \label{def:outputconnectable} A linear compartmental model is \textit{output connectable}  if every compartment has a directed path leading from it to an output 
compartment \citep{GodfreyChapman}. 

\end{definition}

Thus, a monomolecular-reaction-network model $(G, \mathcal{O})$ is output connectable if 
for every species $X_i$ there is a directed path in $G$ from $X_i$ to some output species $X_j \in \mathcal O$.
 Such models are algebraically observable:

\begin{lemma} \label{lem:strongly-connected}
Let $G=(\SS,\CC,\RR)$ be a monomolecular reaction network, 
 and let 
$\mathcal{O} \subseteq  \SS$ be nonempty.  
If $(G, \mathcal O )$ is output connectable, then  $(G, \mathcal{O})$ is algebraically observable.
\end{lemma}

We prove Lemma~\ref{lem:strongly-connected} in Appendix~\ref{sec:appendix}, 
where the lemma is restated as follows: {\em Every output connectable linear compartmental model 
is algebraically observable} (Corollary~\ref{cor:alg-obs}).

\begin{rmk} \label{rmk:struc-obs}
A linear compartmental model is output connectable if and only if it is structurally observable~\citep{GodfreyChapman}.  
Lemma~\ref{lem:strongly-connected} therefore extends this result 
to algebraic observability. 
In fact, for such models, we give 
explicit algebraic-observability relationships for each state variable in terms of inputs, outputs, and parameters (see Proposition~\ref{prop:newy} and its proof).

\end{rmk}

\begin{theorem} \label{thm:sub-super-i-o} 
Let $N_1, \dots, N_p$ be monomolecular networks
with pairwise disjoint sets of species $\SS_1, \ldots, \SS_p$.  
Let 
$\mathcal{O}_i \subseteq  \SS_i$
 be nonempty for $i=1,\dots, p$.  
Assume that, for $i=1,\dots, p-1$, 
the network $N_i$ has
at least one inflow reaction and $(N_i, {\mathcal O}_i)$ is output connectable.
Let $G$ be a network obtained by joining $N_1, \dots, N_p$ by a one-way flow via Scenario 1 or 2. 
Then, if $(N_1,\mathcal{O}_1),$
$\dots,$
$(N_p,\mathcal{O}_p)$ are identifiable, then 
$(G,\mathcal{O}_1 \cup \dots \cup \mathcal{O}_p)$ 
is  identifiable.
\end{theorem}

\begin{proof} This follows directly from Corollary~\ref{cor:join-several} and Lemma~\ref{lem:strongly-connected}. 
\end{proof}

Output connectable models include models arising from strongly connected graphs (more precisely, when the non-flow subnetwork is strongly connected).  See the following examples.

\begin{eg} \label{eg:join-over-0} 
Consider three models, which we call $(N_1, \mathcal{O}_1)$, $(N_2, \mathcal{O}_2)$, and $(N_3, \mathcal{O}_3)$:
\begin{center}
	\begin{tikzpicture}[scale=1.5]
	\draw [black] (-0.5, -.8) rectangle (3.3, .4);
    	\node[] at (0, 0) {0};
    	\node[] at (1, 0) {$X_1$};
    	\node[] at (2, 0) {$X_2$};
    	\node[] at (3, 0) {0};
	 \draw[->] (0.2, 0) -- (0.8, 0);	
	 \draw[->] (1.2, 0.05) -- (1.8, 0.05);	
	 \draw[<-] (1.2, -0.05) -- (1.8, -0.05);	
	 \draw[->] (2.2, 0) -- (2.8, 0);	
 	\draw (0.66,-.49) circle (0.05);	
	 \draw[-] (1, -.15) -- (0.7, -.45);	
%
	\draw [black] (3.6, -.8) rectangle (6.3, .4);
    	\node[] at (4, 0) {0};
    	\node[] at (5, 0) {$X_3$};
    	\node[] at (6, 0) {$0$};
	 \draw[->] (4.2, 0) -- (4.8, 0);	
	 \draw[->] (5.2, 0) -- (5.8, 0);	
 	\draw (4.66,-.49) circle (0.05);	
	 \draw[-] (5, -.15) -- (4.7, -.45);	

	\draw [black] (6.6, -.8) rectangle (9.3, .4);
    	\node[] at (7, 0) {$X_3$};
    	\node[] at (8, 0) {$X_4$};
	\node[] at (9,0) {0};
	 \draw[->] (7.2, 0.05) -- (7.8, 0.05);	
	 \draw[<-] (7.2, -0.05) -- (7.8, -0.05);	
	 \draw[<-] (8.2, 0) -- (8.8, 0);		 
 	\draw (7.66,-.49) circle (0.05);	
	 \draw[-] (8, -.15) -- (7.7, -.45);	
	\end{tikzpicture}
\end{center}
Each model is identifiable \citep{MeshkatSullivantEisenberg}, 
has one inflow reaction, and has strongly connected non-flow subnetwork.
So, by Theorem~\ref{thm:sub-super-i-o}, the model depicted below, which is obtained by joining $N_1$ and $N_2$ via Scenario~1 (by replacing the reactions $X_2 \to 0$ and $0 \to X_3$ by the reaction $X_2 \to X_3$), is also identifiable:
\begin{center}
	\begin{tikzpicture}[scale=1.5]
    	\node[] at (0, 0) {0};
    	\node[] at (1, 0) {$X_1$};
    	\node[] at (2, 0) {$X_2$};
    	\node[] at (3, 0) {$X_3$};
    	\node[] at (4, 0) {$0$};
	 \draw[->] (0.2, 0) -- (0.8, 0);	
	 \draw[->] (1.2, 0.05) -- (1.8, 0.05);	
	 \draw[<-] (1.2, -0.05) -- (1.8, -0.05);	
	 \draw[->] (2.2, 0) -- (2.8, 0);	
	 \draw[->] (3.2, 0) -- (3.8, 0);	
 	\draw (0.66,-.49) circle (0.05);	
	 \draw[-] (1, -.15) -- (0.7, -.45);	
 	\draw (2.66,-.49) circle (0.05);	
	 \draw[-] (3, -.15) -- (2.7, -.45);	
	\end{tikzpicture}
\end{center}
Similarly, by the same theorem, joining $N_1$ and $N_3$ via Scenario 2 (by replacing $X_2 \to 0$ by $X_2 \to X_3$), yields the model $\mathcal M$ displayed earlier in Example~\ref{ex:model-restriction}, which is identifiable.
\end{eg}

The next examples show that partial converses to Theorem~\ref{thm:sub-super-i-o} do {\em not} hold: 
in Scenario~2, if $(G, \mathcal{O}_1 \cup \mathcal{O}_2)$ is identifiable, it does not follow that $N_1$ is identifiable, nor $N_2$.

\begin{eg} \label{eg:why-can-not-conclude-N1-identifiable} 
Consider two models, which we call $(N_1, \mathcal{O}_1)$ and $(N_2, \mathcal{O}_2)$:
\begin{center}
        \begin{tikzpicture}[scale=1.5]
	\draw [black] (-0.4, -.8) rectangle (3.3, .4);
        \node[] at (0, 0) {0};
        \node[] at (1, 0) {$X_1$};
        \node[] at (2, 0) {$X_2$};
        \node[] at (3, 0) {0};
         \draw[->] (0.2, 0.05) -- (0.8, 0.05); 
				 \draw[<-] (0.2, -0.05) -- (0.8, -0.05);
         \draw[->] (1.2, 0.05) -- (1.8, 0.05);  
         \draw[<-] (1.2, -0.05) -- (1.8, -0.05);        
         \draw[->] (2.2, 0) -- (2.8, 0);        
        \draw (0.66,-.49) circle (0.05);        
         \draw[-] (1, -.15) -- (0.7, -.45);    
%
	\draw [black] (3.6, -.8) rectangle (7.3, .4);
    	  \node[] at (4, 0) {0};
        \node[] at (5, 0) {$X_3$};
        \node[] at (6, 0) {$X_4$};
        \node[] at (7, 0) {0};
	 \draw[->] (4.2, -0) -- (4.8, 0); 
	 \draw[->] (5.2, 0.05) -- (5.8, 0.05);	
	 \draw[<-] (5.2, -0.05) -- (5.8, -0.05);	
	 \draw[->] (6.2, 0) -- (6.8, 0);	
 	\draw (4.66,-.49) circle (0.05);	
	 \draw[-] (5, -.15) -- (4.7, -.45);	
  \end{tikzpicture}
\end{center}
Each $N_i$ 
has one inflow reaction and has strongly connected non-flow subnetwork.
The model below, obtained by joining $N_1$ and $N_2$ via Scenario 2 (by replacing $X_2 \to 0$ with $X_2 \to X_3$) is at least locally identifiable \citep{MeshkatSullivantEisenberg}:
\begin{center}
	\begin{tikzpicture}[scale=1.5]
    	\node[] at (0, 0) {0};
    	\node[] at (1, 0) {$X_1$};
    	\node[] at (2, 0) {$X_2$};
    	\node[] at (3, 0) {$X_3$};
    	\node[] at (4, 0) {$X_4$};
	\node[] at (5, 0) {0};
	\node[] at (3,-0.9) {0};
	 \draw[->] (0.2, 0.05) -- (0.8, 0.05);	
	 \draw[<-] (0.2, -0.05) -- (0.8, -0.05);
	 \draw[->] (1.2, 0.05) -- (1.8, 0.05);	
	 \draw[<-] (1.2, -0.05) -- (1.8, -0.05);	
	 \draw[->] (2.2, 0) -- (2.8, 0);	
	 \draw[->] (3.2, 0.05) -- (3.8, 0.05);	
	 \draw[<-] (3.2, -0.05) -- (3.8, -0.05);	
	 \draw[->] (4.2, 0) -- (4.8, 0);	
	 \draw[<-] (3.0, -0.15) -- (3.0, -0.75);	
 	\draw (0.66,-.49) circle (0.05);	
	 \draw[-] (1, -.15) -- (0.7, -.45);	
	\draw (2.56,-.49) circle (0.05);	
	 \draw[-] (2.9, -.15) -- (2.6, -.45);
	\end{tikzpicture}
\end{center}
However, $(N_1, \mathcal{O}_1)$ is unidentifiable~\citep{MeshkatSullivantEisenberg}.  (On the other hand, it is straightforward to check that $(N_2, \mathcal{O}_2)$ is globally identifiable.)

\end{eg}

\begin{eg} \label{eg:why-can-not-conclude-N2-identifiable} 
Consider two models, which we call $(N_1, \mathcal{O}_1)$ and $(N_2, \mathcal{O}_2)$:
\begin{center}
        \begin{tikzpicture}[scale=1.5]
	\draw [black] (-0.4, -.8) rectangle (3.3, .4);
        \node[] at (0, 0) {0};
        \node[] at (1, 0) {$X_1$};
        \node[] at (2, 0) {$X_2$};
        \node[] at (3, 0) {0};
         \draw[->] (0.2, 0) -- (0.8, 0);        
         \draw[->] (1.2, 0.05) -- (1.8, 0.05);  
         \draw[<-] (1.2, -0.05) -- (1.8, -0.05);        
         \draw[->] (2.2, 0) -- (2.8, 0);        
        \draw (0.66,-.49) circle (0.05);        
         \draw[-] (1, -.15) -- (0.7, -.45);    
%
	\draw [black] (3.6, -.8) rectangle (7.3, .4);
    	  \node[] at (4, 0) {0};
        \node[] at (5, 0) {$X_3$};
        \node[] at (6, 0) {$X_4$};
        \node[] at (7, 0) {0};
	 \draw[<-] (4.2, -0) -- (4.8, 0); 
	 \draw[->] (5.2, 0.05) -- (5.8, 0.05);	
	 \draw[<-] (5.2, -0.05) -- (5.8, -0.05);	
	 \draw[->] (6.2, 0.05) -- (6.8, 0.05);	
	 \draw[<-] (6.2, -0.05) -- (6.8, -0.05);	
 	\draw (5.66,-.49) circle (0.05);	
	 \draw[-] (6, -.15) -- (5.7, -.45);	
  \end{tikzpicture}
\end{center}
Each $N_i$ has one inflow reaction, with strongly connected non-flow subnetwork.
The model below, obtained by joining $N_1$ and $N_2$ via Scenario 2 (by replacing $X_2 \to 0$ with $X_2 \to X_3$), is at least locally identifiable \citep{MeshkatSullivantEisenberg}:
\begin{center}
	\begin{tikzpicture}[scale=1.5]
    	\node[] at (0, 0) {0};
    	\node[] at (1, 0) {$X_1$};
    	\node[] at (2, 0) {$X_2$};
    	\node[] at (3, 0) {$X_3$};
    	\node[] at (4, 0) {$X_4$};
	\node[] at (5, 0) {0};
	\node[] at (3,-0.9) {0};
	 \draw[->] (0.2, 0) -- (0.8, 0);	
	 \draw[->] (1.2, 0.05) -- (1.8, 0.05);	
	 \draw[<-] (1.2, -0.05) -- (1.8, -0.05);	
	 \draw[->] (2.2, 0) -- (2.8, 0);	
	 \draw[->] (3.2, 0.05) -- (3.8, 0.05);	
	 \draw[<-] (3.2, -0.05) -- (3.8, -0.05);	
	 \draw[->] (4.2, 0.05) -- (4.8, 0.05);	
	 \draw[<-] (4.2, -0.05) -- (4.8, -0.05);	
	 \draw[->] (3, -0.15) -- (3, -0.75);		 
 	\draw (0.66,-.49) circle (0.05);	
	 \draw[-] (1, -.15) -- (0.7, -.45);	
 	\draw (3.66,-.49) circle (0.05);	
	 \draw[-] (4, -.15) -- (3.7, -.45);	
	\end{tikzpicture}
\end{center}
However, $(N_2, \mathcal{O}_2)$ is unidentifiable~\citep{MeshkatSullivantEisenberg}.  (The model $(N_1, \mathcal{O}_1)$ is globally identifiable, as it is equivalent to the model $(N_2, \mathcal{O}_2)$ in Example~\ref{eg:why-can-not-conclude-N1-identifiable}.)
\end{eg}

In Theorem~\ref{thm:sub-super-i-o}, we saw that if identifiable, output connectable, monomolecular networks $N_i$ are joined by a one-way flow (via Scenario 1 or 2), then the result is still identifiable.
The next main result, Theorem~\ref{thm:onewayflow}, states that if $N_1$ and each of the inductively joined networks $N_1$ and $N_2$, $N_1$ and $N_2$ and $N_3$, etc., are identifiable,
we also conclude that $N_2, N_3, ... $ are identifiable -- 
as long as we are in Scenario~1 and the joining is ``in a row'' over a single reaction.   
In contrast, in Scenario 2, we can not obtain the same conclusion (recall Example~\ref{eg:why-can-not-conclude-N2-identifiable}). 

To prove Theorem~\ref{thm:onewayflow}, we need the following strengthening of \cite[Lemma 3]{MeshkatSullivantEisenberg}.

\begin{prop}[Equations for algebraic observability] \label{prop:newy} 
Let $G= (\mathcal S, \mathcal C, \mathcal R)$ be a monomolecular network, and let $\mathcal O \subseteq \SS$ be nonempty.
Assume that there exists a species 
 $i \in \mathcal S$
 such that for every species $X_j \in \mathcal S \setminus \{X_i\}$, 
 there exists a sequence of reactions $X_j \to \dots \to X_i$ in $G$ 
 from $X_j$ to $X_i$. 
 Then for every such $X_j \in \mathcal S \setminus \{X_i\}$, 
  there exists an 
equation of the form $x_j=g$
that holds (for generic values of the rate constants) along all solutions of 
 $(G,\mathcal O)$, where 
$g$ is a 
$\mathbb{Q}\left( \{ \kappa_{lk} 
	 \mid l \to k \text{ is a reaction in } G  \}\right)$-linear combination 
of 
$x_i$ and the inflow-reaction variables $u_p$ (for inflow reactions $0 \to X_p$) and their derivatives $x_i^{(q)}$ and $u_p^{(q)}$, and 
the coefficient of at least one of the $x_i^{(q)}$'s is nonzero.  
\end{prop}
\noindent
We prove Proposition~\ref{prop:newy} in the appendix.

The next result pertains to networks joined by a one-way flow ``in a row''.  For networks $N_1, \dots, N_p$ joined by a one-way flow, we say they are joined {\em in a row} if the new reactions are from $N_1$ to $N_2$, from $N_2$ to $N_3$, and so on; more precisely, the joining functions 
$\phi_q: \mathfrak{I}_q \to \{i \mid X_i \in  {\mathcal S}_{q+1}\cup \dots \cup {\mathcal S}_{p} \}$ (for $q=1,\dots, p-1$) satisfy $\phi( \mathfrak{I}_q) \subseteq \{i \mid X_i \in  {\mathcal S}_{q+1} \}$.  

We also require a stronger condition than output connectable, where each of the networks formed by joining $N_1$, $N_2$, ..., $N_k$, for $k=1,...,p-1$, is output connectable, which can be considered as \textit{inductively output connectable}. 

Additionally, we consider the following version of identifiability: 
a model $\mathcal M$ obtained by joining two models obtained by joining $(N_1,\mathcal{O}_1)$ and $(N_2,\mathcal{O}_2)$ by a one-way flow over a single reaction 
is {\em identifiable after substitution} if the model is identifiable when, for each output variable, the input-output equation is taken as in Definition~\ref{defn:identify}, or the corresponding one from~\eqref{eq:i-o-for-general-model} for $\mathcal M$, or -- for outputs in $\mathcal{O}_2$, is obtained by taking the corresponding input-output equation in~\eqref{eq:i-o-for-general-model} for $(N_2,\mathcal{O}_2)$ and then substituting an expression for the inflow rate in $N_2$ (for the unique inflow reaction that is replace in $\widetilde M$) that is valid along all trajectories of $\widetilde M$.
We again do not know whether this (possibly stronger) version of identifiability encompasses fewer models than Definition~\ref{defn:identify}.  Also, although checking whether a model is identifiable after substitution is difficult, our results only pertain to finitely many input-output equations, those arising as in~\eqref{eq:i-o-subbed} in the following proof.

\begin{theorem} \label{thm:onewayflow} 
Let $G$ be a network obtained by joining, in a row, monomolecular networks 
$N_1, \dots, N_p$ with pairwise disjoint sets of species $\SS_1, \ldots, \SS_p$ by a one-way flow -- but only via Scenario~1.  
Let $\mathcal{O}_1 \subseteq  \SS_1,\dots, \mathcal{O}_p \subseteq  \SS_p$
 be nonempty. 
Assume the following:
\begin{enumerate}
\item each joining by a one-way flow is over a single reaction,
\item every $N_i$ (for $i=1, \dots, p$) has at least one inflow reaction, 
\item for every $X_{\ell} \in \mathcal{O}_i$ (for any $i=1, \dots, p$) 
there is a directed path in $N_i$ from an inflow-reaction (input) species to $X_\ell$, 
\item for $q=1, \dots, p-1$, 
there exists a species $X_{i_q} \in \mathcal{O}_q$
 such that 
 for 
every species $X_j \in  \mathcal{S}_1 \cup \dots \cup \mathcal{S}_q \setminus \{ X_{i_q} \}$, 
there exists a sequence of reactions $X_j \to \dots \to X_{i_q}$ in $G$ 
from $X_j$ to $X_{i_q}$, 
\item the following $p-1$ models are identifiable after substitution: $(N_1,\mathcal{O}_1)$, the model 
obtained by joining $(N_1,\mathcal{O}_1)$ and $(N_2,\mathcal{O}_2)$, 
..., and the model obtained by joining $(N_1,\mathcal{O}_1)$, $(N_2,\mathcal{O}_2)$,  ... $(N_{p-1},\mathcal{O}_{p-1})$ (via the same joining functions as for $G$). 
\end{enumerate}
Then if
$(N_2,\mathcal{O}_2),$
$\dots,$
$(N_p,\mathcal{O}_p)$ are all identifiable, then 
$(G,\mathcal{O}_1 \cup \dots \cup \mathcal{O}_p)$ 
is identifiable. 
Conversely, if 
$(G,\mathcal{O}_1 \cup \dots \cup \mathcal{O}_p)$ 
is identifiable after substitution,  then 
$(N_2,\mathcal{O}_2),$
$\dots,$
$(N_p,\mathcal{O}_p)$ are identifiable.
\end{theorem}

\begin{proof} 
The forward direction (``$\Rightarrow$'') follows from Theorem~\ref{thm:sub-super-i-o}.  

For the backward direction (``$\Leftarrow$''), assume that 
$(G,\mathcal{O}_1 \cup \dots \cup \mathcal{O}_p)$ 
is identifiable. 
We prove by induction that $(N_2,\mathcal{O}_2),$
$\dots,$
$(N_p,\mathcal{O}_p)$ are identifiable.  By assumption $(N_1,\mathcal{O}_1)$ is identifiable.  So, for induction, assume that $(N_{r-1},\mathcal{O}_{r-1})$ is identifiable for some $2 \leq r \leq p$.  We must show that  $(N_{r},\mathcal{O}_{r})$ is identifiable.

The $N_i$'s are joined ``in a row'', so we let $M$ denote the network obtained by joining $N_1, \dots, N_{r-1}$ by a one-way flow, 
and let $\widetilde{M}$ be obtained from joining $M$ and $N_r$
(via the same joining functions as for $G$).  
By hypothesis, $\widetilde{M}$ is obtained from joining $M$ and $N_r$ over a single reaction: for some species $X_i$ and $X_{j'}$, the outflow reaction $X_i \to 0$ in $M$ and the inflow (input) reaction $0 \to X_{j'}$ are replaced by the new reaction $X_i \to X_{j'}$.
Also by hypothesis,  $(\widetilde{M},\mathcal{O}_1 \cup \dots \cup \mathcal{O}_{r})$ is identifiable.  

Let $n$ and $m$ denote the number of species of, respectively, $\widetilde{M}$ and $M$.  
Following the proof of Theorem~\ref{thm:join-new-rxn-iden}, specifically, from equation~\eqref{eq:odenetwork2}, the ODEs of $N_r$ are as follows:
\begin{align}\label{eq:odenetwork2-again}
\begin{pmatrix}
x'_{m+1} \\ 
\vdots \\
x'_n
\end{pmatrix}
~ = ~
g(\gamma, u^{(2)} ;~ x_{m+1},\dots, x_n) + 
 \widetilde{u}^{(2)}_{0 \to X_{j'} }{\bf e}_{j'}~,
\end{align}
where $\gamma$ is the input vector of non-inflow rate constants, 
and 
$\widetilde{u}^{(2)}_{0 \to X_{j'}}$ 
is the rate for the reaction $0 \to X_{j'}$ 
and $u^{(2)}=u^{(2)}(t)$ is the vector of all remaining inflow rates.

Similarly, using equation~\eqref{eq:last-eqns-G}, the last $(n-m)$ ODEs of $\widetilde{M}$ are: 
\begin{align} \label{eq:last-eqns-G-again}
	\begin{pmatrix}
	x'_{m+1} \\ 
	\vdots \\
	x'_n
	\end{pmatrix}
~& = ~
g(\gamma, u^{(2)} ;~x_{m+1},\dots, x_n) +  \kappa_{i0} x_i {\bf e}_{j'}  ~.
\end{align}
Here, $\kappa_{i0}$
denotes the rate constant for the 
outflow reaction $X_i \to 0$ in $M$.

By assumption, 
there exists 
$X_{i_{r-1}} \in \mathcal{O}_{r-1}$
such that for 
  every species $X_j \in \mathcal{S}_1 \cup\dots \cup \mathcal{S}_{r-1} \setminus \{ X_{i_{r-1}} \}$, 
 there exists a sequence of reactions $X_j \to \dots \to X_{i_{r-1}}$ in $\widetilde{M}$ (and thus in $M$)
 from $X_j$ to $X_{i_{r-1}}$. 
Hence, $M$ and $X_{i_{r-1}}$ together satisfy the hypotheses of Proposition~\ref{prop:newy}.

Thus, 
there exists an 
equation of the form 
$x_i=g_i$ 
that holds (for generic choices of the rate constants) 
along solutions of $M$, where $g_i$ 
is a 
$\mathbb{Q}\left( \{ \kappa_{l k} 
	\mid l \to k \text{ is a reaction in } M  \}\right)$-linear combination 
of 
$z_{i_{r-1} } = x_{i_{r-1}}$
and the inflow-reaction variables and their derivatives, and the coefficient of at least one 
$z_{i_{r-1}}^{(q)}$ 
is nonzero.  
Thus, from equations~\eqref{eq:odenetwork2-again} and~\eqref{eq:last-eqns-G-again}, when we make the following substitution
into the ODEs of $N_r$:
	\begin{align} \label{eq:sub}
	\widetilde{u}^{(2)}_{0 \to X_{j'}} ~:=~ \kappa_{i0} g_i ~,
	\end{align}	
we get differential equations 
satisfied by solutions of the dynamical system defined by $\widetilde{M}$. 
 
Hence, any input-output equation for $N_r$ can be transformed into an input-output equation for $\widetilde{M}$ by making the substitution~\eqref{eq:sub}.
Specifically, when we make this substitution into 
 the following input-output equations for $N_r$ (one for each $X_{\ell} \in \mathcal{O}_r$) from Proposition~\ref{prop:ioscc} (which applies because
 of hypothesis (3) in the statement of Theorem~\ref{thm:onewayflow}):
	\begin{align} \label{eq:i-o-not-yet-subbed}
 	\det (\partial I -{A}_{H_{\ell}}) z_{\ell} ~&=~  
		 \sum_{j \in In_{r} \cap V_{H_{\ell}} } (-1)^{{\ell}+j} \det \left( \partial I-{A}_{H_{\ell}} \right)_{j{\ell}} u_j   ~,
	\end{align} 
we obtain the following input-output equations for $\widetilde{M}$ 
(one for each $\ell \in \mathcal{O}_r$):
	\begin{align} \label{eq:i-o-subbed}
 	\det (\partial I -{A}_{H_{\ell}}) z_{\ell} ~&=~  
		 \sum_{j \in (In_{r} \cap V_{H_{\ell}}) \setminus \phi(\mathfrak I_{r-1}) } (-1)^{{\ell}+j} \det \left( \partial I-{A}_{H_{\ell}} \right)_{j{\ell}} u_j  \\ \notag
		& \quad \quad \quad \quad  +~
		 (-1)^{{\ell}+j'} \det \left( \partial I-{A}_{H_{\ell}} \right)_{j' {\ell}}  \kappa_{i0} g_i ~,
	\end{align}		
where
$H_{\ell}=(V_{H_{\ell}}, E_{H_{\ell}})$ is the output-reachable subgraph 
(of the directed graph underlying the non-flow subnetwork of $N_r$) to $\ell$, and $A_{H_{\ell}}$ is the corresponding compartmental matrix.  Also, $In_r$ denotes the set of all inflow species in $N_r$.

Next, we claim that the input-output equations for $M$ obtained from 
	Proposition~\ref{prop:ioscc}
	are also input-output equations for $\widetilde{M}$.
	Indeed, there are no reactions in $\widetilde{M}$ 
	from outside of $M$ into $M$, so for any output variable $X_{i^*} \in \mathcal{O}_1\cup \dots \cup \mathcal{O}_{r-1}$ in $M$, the output-reachable subgraph (of $\widetilde{M}$) to $X_{i^*}$ is contained in $M$.  Thus, our claim follows from Proposition~\ref{prop:ioscc}.
	 
Thus, the following are $|\mathcal{O}_1|+\dots + |\mathcal{O}_r|$ input-output equations for $\widetilde{M}$:
	\begin{enumerate}
	\item the $|\mathcal{O}_1|+\dots + |\mathcal{O}_{r-1}|$ input-output equations for $M$ obtained from 
	Proposition~\ref{prop:ioscc} (which are the same equations $\widetilde M$ obtained from 
	Proposition~\ref{prop:ioscc} for those outputs in $\mathcal{O}_1 \cup \dots \cup \mathcal{O}_{r-1}$), 
and 
	\item the $|\mathcal{O}_r|$ equations in~\eqref{eq:i-o-subbed}.
	\end{enumerate}
These input-output equations are algebraically independent, because they each involve a distinct output.
Also, by the ``identifiable after substitution'' assumption, the equations in~\eqref{eq:i-o-subbed} can be used to assess identifiability (after substitution). 
Thus, as we have 
$|\mathcal{O}_1  \cup \dots \cup \mathcal{O}_r|  =  |\mathcal{O}_1|+\dots + |\mathcal{O}_r|$
algebraically independent input-output equations, we get a coefficient map for $\widetilde M$, which we denote by $c=(c^M; c^{(r)})$.
By hypothesis, $c=(c^M ; c^{(r)})$ is finite-to-one.  Here
and in the remainder of this proof, we write ``finite-to-one'' to mean ``generically finite-to-one (respectively, generically one-to-one)''.

Let $c^{N_r}$ denote the coefficient map for $N_r$ arising from the input-output equations~\eqref{eq:i-o-not-yet-subbed}.  We claim that $(c^M; c^{N_r} ) $ is finite-to-one.  Indeed, comparing 
equations~\eqref{eq:i-o-not-yet-subbed} and~\eqref{eq:i-o-subbed}, we see that 
for each coefficient in (the expansion of) equation~\eqref{eq:i-o-not-yet-subbed} (i.e., each coordinate of $c^{N_r}$), 
either this coefficient also appears as a coefficient in~\eqref{eq:i-o-subbed}, 
or a (nonzero) $\mathbb{F}$-multiple of it is a coefficient of some $x_{i_{r-1}}^{(q)}$ in~\eqref{eq:i-o-subbed}, where $\mathbb{F}:=\mathbb{Q}\left( \{ \kappa_{lk} \mid l \to k \text{ is a reaction in } M  \}\right)$. 
Conversely, each coefficient in (the expansion of) equation~\eqref{eq:i-o-subbed} (i.e., each coordinate of $c^{(r)}$), if not also a coordinate of $c^{N_r}$, is an $\mathbb{F}$-multiple of a coordinate in $c^{N_r}$.  
From generic input-output data, any rational function in $\mathbb{F}$ can be recovered (up to finitely many values) using $c^{M}$, and so 
the fact that $(c^{M}; c^{(r)} ) $ is finite-to-one
implies that 
$(c^{M}; c^{N_r} ) $ is finite-to-one, as we claimed.

The function $c^{M}$ depends only on the parameters in $M$, and similarly $ c^{N_r}$ depends only on the parameters in $N_r$.  So, 
the fact that $(c^{M}; c^{N_r} ) $ is finite-to-one implies that 
$ c^{N_r}$ is finite-to-one.  Hence, $N_r$ is identifiable. 
\end{proof}

\begin{eg} 
In Example \ref{eg:join-over-0}, the model formed by joining $N_1$ with $N_2$ is identifiable.  We also know that $N_1$ is identifiable.
Hence, by Theorem~\ref{thm:onewayflow}, 
$N_2$ 
 is also identifiable.
\end{eg}

\subsubsection{Joining strongly connected, monomolecular networks via Scenario 3 or 4} 
In this subsection, we show that joining certain monomolecular networks by new reactions -- namely, strongly connected networks without leaks -- preserves identifiability (Theorem~\ref{thm:super-i-o}).

\begin{theorem} \label{thm:super-i-o} 
Let $N_1, \dots, N_p$ be monomolecular networks
with pairwise disjoint sets of species $\SS_1, \ldots, \SS_p$.  
Let $\mathcal{O}_1 \subseteq  \SS_1,\dots, \mathcal{O}_p \subseteq  \SS_p$
 be nonempty. 
Assume, for $i=1, \dots, p-1$, 
that $N_i$ has no outflows and at least one inflow reaction, and that 
the non-flow subnetwork of $N_i$ is strongly connected.
Let $G$ be 
obtained by joining $N_1, \dots, N_p$ by a one-way flow 
 via Scenario 3 or 4.  
Assume, moreover, that each joining by a one-way flow is over a single reaction. 
Then,  if $(N_1,\mathcal{O}_1),$
$\dots,$
$(N_p,\mathcal{O}_p)$ are all identifiable, then 
$(G,\mathcal{O}_1 \cup \dots \cup \mathcal{O}_p)$ 
is identifiable.  
\end{theorem}

\begin{proof} 
For $i=1,\dots, p-1$, let $\widetilde{N}_i$ denote the network obtained from $N_i$ by adding an outflow reaction (leak) at the compartment 
from which a new one-way-flow reaction emerges in $G$.
By construction, $G$ is obtained by joining $\widetilde{N}_1,\dots, \widetilde{N}_{p-1}$, and $N_p$ by Scenario 1 or~2.  

Assume that $(N_1,\mathcal{O}_1),$
$\dots,$
$(N_p,\mathcal{O}_p)$ are identifiable.  
Then (for $i=1,\dots, p-1$), by Lemma~\ref{lem:add-1-leak}, 
the model  $(\widetilde{N}_i,\mathcal{O}_i)$ is identifiable (here we use the fact that $N_i$ has no outflow reactions and is strongly connected).
So, by Theorem~\ref{thm:sub-super-i-o}, $(G,\mathcal{O}_1 \cup \dots \cup \mathcal{O}_p)$ 
is identifiable.
\end{proof}

\begin{eg} 

Both linear compartmental models below are at least locally identifiable~\citep{MeshkatSullivantEisenberg}:
\begin{center}
	\begin{tikzpicture}[scale=1.5]
	\draw [black] (-0.4, -.7) rectangle (2.3, .4);
    	\node[] at (0, 0) {0};
    	\node[] at (1, 0) {$X_1$};
    	\node[] at (2, 0) {$X_2$};
	 \draw[->] (0.2, 0) -- (0.8, 0);	
	 \draw[->] (1.2, 0.05) -- (1.8, 0.05);	
	 \draw[<-] (1.2, -0.05) -- (1.8, -0.05);	
 	\draw (0.66,-.49) circle (0.05);	
	 \draw[-] (1, -.15) -- (0.7, -.45);	
%
	\draw [black] (3.6, -.7) rectangle (6.3, .4);
    	\node[] at (4, 0) {$X_3$};
    	\node[] at (5, 0) {$X_4$};
	\node[] at (6,0) {0};
	 \draw[->] (4.2, 0.05) -- (4.8, 0.05);	
	 \draw[<-] (4.2, -0.05) -- (4.8, -0.05);	
	 \draw[<-] (5.2, 0) -- (5.8, 0);		 
 	\draw (4.66,-.49) circle (0.05);	
	 \draw[-] (5, -.15) -- (4.7, -.45);	
	\end{tikzpicture}
\end{center}
Thus, by Theorem \ref{thm:super-i-o}, joining the networks by Scenario 3 yields a model, below, that is at least locally identifiable:  
\begin{center}
	\begin{tikzpicture}[scale=1.5]
    	\node[] at (0, 0) {0};
    	\node[] at (1, 0) {$X_1$};
    	\node[] at (2, 0) {$X_2$};
	 \draw[->] (0.2, 0) -- (0.8, 0);	
	 \draw[->] (1.2, 0.05) -- (1.8, 0.05);	
	 \draw[<-] (1.2, -0.05) -- (1.8, -0.05);	
 	\draw (0.66,-.49) circle (0.05);	
	 \draw[-] (1, -.15) -- (0.7, -.45);	
	 \draw[->] (2.2, 0) -- (2.8, 0);	
    	\node[] at (3, 0) {$X_3$};
    	\node[] at (4, 0) {$X_4$};
	\node[] at (5,0) {0};
	 \draw[->] (3.2, 0.05) -- (3.8, 0.05);	
	 \draw[<-] (3.2, -0.05) -- (3.8, -0.05);	
	 \draw[<-] (4.2, 0) -- (4.8, 0);		 
 	\draw (3.66,-.49) circle (0.05);	
	 \draw[-] (4, -.15) -- (3.7, -.45);	
	\end{tikzpicture}
\end{center}

\end{eg}

\begin{eg}
Like Theorem~\ref{thm:sub-super-i-o} earlier, Theorem~\ref{thm:super-i-o} can {\em  not} be extended to conclude that, if $(G, \mathcal{O}_1 \cup \dots \cup \mathcal{O}_p)$ is identifiable, then $(N_2, \mathcal{O}_2), \dots, (N_p, \mathcal{O}_p)$ are also.  We can see this by modifying Example~\ref{eg:why-can-not-conclude-N2-identifiable}.  In that example, we saw that the following model is locally identifiable: 
\begin{center}
	\begin{tikzpicture}[scale=1.5]
    	\node[] at (0, 0) {0};
    	\node[] at (1, 0) {$X_1$};
    	\node[] at (2, 0) {$X_2$};
    	\node[] at (3, 0) {$X_3$};
    	\node[] at (4, 0) {$X_4$};
	\node[] at (5, 0) {0};
	\node[] at (3,-0.9) {0};
	 \draw[->] (0.2, 0) -- (0.8, 0);	
	 \draw[->] (1.2, 0.05) -- (1.8, 0.05);	
	 \draw[<-] (1.2, -0.05) -- (1.8, -0.05);	
	 \draw[->] (2.2, 0) -- (2.8, 0);	
	 \draw[->] (3.2, 0.05) -- (3.8, 0.05);	
	 \draw[<-] (3.2, -0.05) -- (3.8, -0.05);	
	 \draw[->] (4.2, 0.05) -- (4.8, 0.05);	
	 \draw[<-] (4.2, -0.05) -- (4.8, -0.05);	
	 \draw[->] (3, -0.15) -- (3, -0.75);		 
 	\draw (0.66,-.49) circle (0.05);	
	 \draw[-] (1, -.15) -- (0.7, -.45);	
 	\draw (3.66,-.49) circle (0.05);	
	 \draw[-] (4, -.15) -- (3.7, -.45);	
	\end{tikzpicture}
\end{center}
This model is formed by joining the following models, 
$(N_1', \mathcal{O}_1')$ and $(N_2, \mathcal{O}_2)$, by Scenario 3:
\begin{center}
        \begin{tikzpicture}[scale=1.5]
	\draw [black] (-0.4, -.7) rectangle (2.3, .4);
        \node[] at (0, 0) {0};
        \node[] at (1, 0) {$X_1$};
        \node[] at (2, 0) {$X_2$};
         \draw[->] (0.2, 0) -- (0.8, 0);        
         \draw[->] (1.2, 0.05) -- (1.8, 0.05);  
         \draw[<-] (1.2, -0.05) -- (1.8, -0.05);        
        \draw (0.66,-.49) circle (0.05);        
         \draw[-] (1, -.15) -- (0.7, -.45);    
%
	\draw [black] (3.7, -.7) rectangle (7.3, .4);
    	  \node[] at (4, 0) {0};
        \node[] at (5, 0) {$X_3$};
        \node[] at (6, 0) {$X_4$};
        \node[] at (7, 0) {0};
	 \draw[<-] (4.2, -0) -- (4.8, 0); 
	 \draw[->] (5.2, 0.05) -- (5.8, 0.05);	
	 \draw[<-] (5.2, -0.05) -- (5.8, -0.05);	
	 \draw[->] (6.2, 0.05) -- (6.8, 0.05);	
	 \draw[<-] (6.2, -0.05) -- (6.8, -0.05);	
 	\draw (5.66,-.49) circle (0.05);	
	 \draw[-] (6, -.15) -- (5.7, -.45);	
  \end{tikzpicture}
\end{center}
As noted earlier in Example~\ref{eg:why-can-not-conclude-N2-identifiable}, model $(N_2, \mathcal{O}_2)$ is unidentifiable.
\end{eg}

 Our final theorem in this section is a partial converse to Theorem~\ref{thm:super-i-o}: If $N_1$ and each of the inductively joined networks $N_1$ and $N_2$, $N_1$ and $N_2$ and $N_3$, etc., are all identifiable (in Scenario 4), then each $N_i$ is identifiable.

\begin{theorem} \label{thm:super-i-o-iff} 
Let $G$ be a network obtained by joining, in a row, monomolecular networks 
$N_1, \dots, N_p$ with pairwise disjoint sets of species $\SS_1, \ldots, \SS_p$ by a one-way flow -- but only via Scenario~4. Let $\mathcal{O}_1 \subseteq  \SS_1,\dots, \mathcal{O}_p \subseteq  \SS_p$
 be nonempty. 
Assume the following:
\begin{enumerate}

\item each joining by a one-way flow is over a single reaction,
\item every $N_i$ (for $i=1, \dots, p$) has at least one inflow reaction,
\item for $i=1, \dots, p-1$, the network $N_i$ has no outflows and the non-flow subnetwork of $N_i$ is strongly connected, 
\item for every $\ell \in \mathcal{O}_p$, 
there is a directed path in $N_p$ from an inflow-reaction species (input) to $X_{\ell}$, 
\item the following $p-1$ models are identifiable after substitution: 
$(N_1,\mathcal{O}_1)$, 
the model obtained by joining $(N_1,\mathcal{O}_1)$ and $(N_2,\mathcal{O}_2)$, 
..., and the model obtained by joining $(N_1,\mathcal{O}_1)$, $(N_2,\mathcal{O}_2)$, ... , and $(N_{p-1},\mathcal{O}_{p-1})$ 
(via Scenario 4~and the same joining functions as for $G$). 
\end{enumerate}
Then 
if
$(N_2,\mathcal{O}_2),$
$\dots,$
$(N_p,\mathcal{O}_p)$ are all identifiable, then 
$(G,\mathcal{O}_1 \cup \dots \cup \mathcal{O}_p)$ 
is identifiable. 
Conversely, if 
$(G,\mathcal{O}_1 \cup \dots \cup \mathcal{O}_p)$ 
is identifiable after substitution, then 
$(N_2,\mathcal{O}_2),$
$\dots,$
$(N_p,\mathcal{O}_p)$ are identifiable.
  \end{theorem}

\begin{proof}
The forward direction (``$\Rightarrow$'') follows from Theorem~\ref{thm:super-i-o}.

We now prove the backward direction (``$\Leftarrow$''). 
For $i=1,\dots, p-1$, let $\widetilde{N}_i$ denote the network obtained from $N_i$ by adding an outflow reaction (leak) at the compartment 
from which a new one-way-flow reaction emerges in $G$.  It follows, by construction, that 
for $1 \leq j < k \leq p$, 
the model obtained by joining 
$(N_j,\mathcal{O}_j)$, $(N_{j+1},\mathcal{O}_{j+1})$, ... , and $(N_{k},\mathcal{O}_{k})$ (via Scenario 4~and the same joining functions as for $G$) equals the model 
obtained by joining $(\widetilde{N}_j,\mathcal{O}_j)$, $(\widetilde{N}_{j+1},\mathcal{O}_{j+1})$, ... ,$(\widetilde{N}_{k-1},\mathcal{O}_{k-1})$, and $(N_{k},\mathcal{O}_{k})$ via Scenario 1 (and the same joining functions as for $G$).  We use this fact below.

We prove the following (stronger) claim: {\em For $i=1,\dots,p$, 
	\begin{enumerate}
	\item[(a)] the model $(N_i, \mathcal{O}_i)$ is identifiable (and hence, by Lemma~\ref{lem:add-1-leak}, $(\widetilde{N}_{i},\mathcal{O}_i)$ also is, if $i \leq p-1$), and 
	\item[(b)] if $i \leq p-1$, the model $(M_i, \mathcal{O}_{1} \cup \dots \cup \mathcal{O}_{i})$ is identifiable, where $M_i$ denotes the network obtained by joining $(\widetilde{N}_{1},\mathcal{O}_{1}), \dots, (\widetilde{N}_{i},\mathcal{O}_{i})$ via Scenario 1 (and the same joining functions as for $G$).
 	\end{enumerate}
 }
 We prove this claim by strong induction on $i$.  For the base case, $i=1$, part (a) holds by assumption, and (b) follows, as noted above, from Lemma~\ref{lem:add-1-leak}.
 
 For the inductive hypothesis, assume that (a) and (b) hold for $i=1,2,\dots, m-1$ for some $2 \leq m \leq p$.  We prove the $i=m$ case of the claim by showing that 
  Theorem~\ref{thm:onewayflow} 
   (the ``$\Leftarrow$'' direction) applies 
   to the networks $\widetilde{N}_1, \dots, \widetilde{N}_{i-1}$, and $N_i$.
   As noted above, the network obtained by joining the networks 
  $(\widetilde{N}_1,\mathcal{O}_1)$, $(\widetilde{N}_{2},\mathcal{O}_{2})$, ... , $(\widetilde{N}_{i-1},\mathcal{O}_{i-1})$, and $(N_{i},\mathcal{O}_{i})$ by Scenario 1 
  equals the network obtained by joining $(N_1,\mathcal{O}_1)$, ... , and $(N_{i},\mathcal{O}_{i})$ 
  by Scenario 4, which by hypothesis is identifiable.  
  Also, by the inductive hypothesis,  the models  $(M_1, \mathcal{O}_{1} ), \dots,  (M_{i-1}, \mathcal{O}_{1} \cup \dots \cup \mathcal{O}_{i-1})$ are all identifiable.    
  Finally, hypotheses (3) and (4) in the statement of Theorem~\ref{thm:onewayflow} apply to 
  the networks 
  $\widetilde{N}_1, \dots, \widetilde{N}_{i-1}$, and $N_i$, because of hypotheses (3) and (4) in the statement of Theorem~\ref{thm:super-i-o-iff}.  
Therefore, Theorem~\ref{thm:onewayflow} 
   (the ``$\Leftarrow$'' direction) applies, 
   and so $N_i$ is identifiable.  This verifies (a).  
    
    For (b), assume 
     $i \leq p-1$. By part (a) of the inductive hypothesis,  the networks  $\widetilde{N}_1, \dots, \widetilde{N}_{i}$ are identifiable.  Hence, by Theorem~\ref{thm:sub-super-i-o}, the model $(M_i, \mathcal{O}_{1} \cup \dots \cup \mathcal{O}_{i})$  is identifiable.  
\end{proof}

Strongly connected networks are output connectable, so 
we obtain the following unifying corollary to 
 Theorems \ref{thm:sub-super-i-o}, \ref{thm:onewayflow}, \ref{thm:super-i-o}, and \ref{thm:super-i-o-iff}. 

\begin{corollary} \label{cor:strongly-connected}
Let $N_1, \dots, N_p$ be monomolecular networks
with pairwise disjoint sets of species $\SS_1, \ldots, \SS_p$.  
Assume, for $i=1,\dots,p$, that $N_i$ has at least one inflow reaction 
and that the non-flow subnetwork of $N_i$ is strongly connected. Let  
$\mathcal{O}_i \subseteq  \SS_i$
 be nonempty for $i=1,\dots, p$. 

\begin{enumerate}

\item Let $G$ be a network obtained by joining $N_1, \dots, N_p$ by a one-way flow via Scenario 1, 2, 3, or 4.
If the joining is by Scenario 3 or 4, assume additionally that each joining 
is over a single reaction, and that, for  $i=1, \dots, p-1$, the network $N_i$ has no outflows. 
Then, if $(N_1,\mathcal{O}_1),$
$\dots,$
$(N_p,\mathcal{O}_p)$ are identifiable, then 
$(G,\mathcal{O}_1 \cup \dots \cup \mathcal{O}_p)$ 
is identifiable.

\item Let $G$ be a network obtained by joining, in a row, $N_1, \dots, N_p$ by a one-way flow via Scenario~1 or 4.   Assume that each joining is over a single reaction and that
the following models are identifiable after substitution: 
$(N_1,\mathcal{O}_1)$, 
the model obtained by joining $(N_1,\mathcal{O}_1)$ and $(N_2,\mathcal{O}_2)$, 
..., and the model obtained by joining $(N_1,\mathcal{O}_1)$, ... , and $(N_{p-1},\mathcal{O}_{p-1})$ (via the same joining functions as for $G$). If the joining is by Scenario 4, assume that, for  $i=1, \dots, p-1$, the network $N_i$ has no outflows.
If
$(N_2,\mathcal{O}_2),$
$\dots,$
$(N_p,\mathcal{O}_p)$ are all identifiable, then 
$(G,\mathcal{O}_1 \cup \dots \cup \mathcal{O}_p)$ 
is identifiable. 
Conversely, if 
$(G,\mathcal{O}_1 \cup \dots \cup \mathcal{O}_p)$ 
is identifiable after substitution, then 
$(N_2,\mathcal{O}_2),$
$\dots,$
$(N_p,\mathcal{O}_p)$ are identifiable.
 \end{enumerate}

\end{corollary}

\section{Steady-State Invariants} \label{sec:elim}

In this section, we move away from identifiability and toward the problem of understanding how steady-state invariants of networks obtained by gluing are related to the steady-state invariants of the joined networks before gluing. Steady-state invariants are polynomial equations satisfied by the 
species concentrations at steady state \citep{Gunawardena,Manrai}.  These polynomials are used for model comparison and are particularly useful when only incomplete data are available \citep{Harrington,NAG,MacLean}.  Specifically, when only some of the species concentrations are measurable,  
an ideal obtained 
by eliminating non-measurable species variables from the steady-state equations is computed, 
and then the generators of this ideal are used to test goodness-of-fit.

However, eliminating the unobservable variables to obtain a set of steady-state invariants 
can be computationally challenging, and the resulting Gr\"obner basis, when it can be computed, is often large and difficult to interpret.  One of our aims, therefore, is to determine how the steady-state invariants of a large network can be built from those of smaller subnetworks.

Our progress toward this aim is as follows.  Consider a network $N$ obtained by gluing two networks $N_1$ and $N_2$ over complexes or reactions.  
We are interested in determining how the steady-state invariants of $N$, after projecting them to involve only species and reactions in $N_i$, are related to the steady-state invariants of $N_i$.  
First, we show that 
every such projection is a steady-state invariant of
  $N_i$ 
(Proposition~\ref{prop:set-relationship}).  
However, in general, 
some of the steady-state invariants of $N_i$ are {\em not}
projected steady-state invariants of~$N$.  
This motivates us to find conditions when 
all steady-state invariants of $N_i$ arise as projections. 
We succeed 
for certain monomolecular networks obtained by gluing two networks over a species (Theorem~\ref{thm:mono-single-species}) or a single reaction (Theorem~\ref{thm:mono-reaction}). 
Moreover, in the case of monomolecular networks glued over a species, under some hypotheses, we recover the entire elimination ideal from the elimination ideals of the smaller networks $N_i$ (Theorem~\ref{thm:mono-glue-species-invt}).

\subsection{Connection to related work}
Steady-state invariants are not the only situation in which 
species variables of a reaction network are eliminated.  
Another context arises when analyzing a network's steady states, specifically its capacity for multistationarity.  
Here certain variables (usually intermediate complexes) often can be eliminated (usually linearly) so that there are effectively fewer steady-state equations to solve~\citep{perspective}.  
\citet{TG}
performed such eliminations for post-translational modification networks, and subsequently Feliu and Wiuf and co-authors extended these ideas to signaling 
networks~\citep{feliu2018,feliu-signaling,feliu-wiuf-crn,
feliu-wiuf-ptm}, gave graphical criteria for when such elimination succeeds \citep{saez-gph}, and proved that the Gr\"obner basis of the steady-state ideal of a network extends to one that includes intermediates \citep{sadeghimanesh2018groebner}.

Here we too are interested in eliminating species from the steady-state equations that are experimentally unobservable.  However, our setup and the questions we ask differ from those in the above references.  For us, the set of variables to eliminate is given, and we would like to know how joining networks affects these eliminations.
Earlier authors, in contrast, focused on eliminating as many species as possible. 

A second situation involving elimination in reaction networks pertains to quasi-steady state and other approximations~\citep{QSSA,sweeney}.  Here, elimination is performed to obtain a lower-dimensional approximation of the system, which is valid when certain assumptions on the rate constants are met~\citep{guna-linear}.  In our work, however, we are interested in steady states of the full system, not a reduced system.

\subsection{Setup}

We begin by introducing steady-state ideals and steady-state invariants.
\begin{defn}
Let $N=(\calS, \calC, \calR)$ be a network with mass-action ODEs:
\[
\frac{dx}{dt} ~=~  \sum_{ y_i \to y_j~ {\rm is~in~} \RR} \kappa_{ij} x^{y_i}(y_j - y_i)  
~=:~  \left( f_1(x), f_2(x), \dots, f_n(x)  \right)~.
\]  
We call $f_1,f_2,\dots, f_n$ the \emph{system polynomials} of $N$, and they generate the 
 \emph{steady-state ideal}: 
$$I_{N} ~:=~ \Big\langle
	f_1(x),~ f_2(x),~ \dots,~ f_n(x) 
 \Big\rangle ~\subseteq~ \Q[ {\boldsymbol \kappa};{\bf x}]~.$$
\end{defn}

Every $g \in I_N$ vanishes at steady state and so we say that $g$ is a {\em steady-state invariant}.  As mentioned earlier, we are interested in steady-state invariants that involve certain observable variables ${\bf x}_{j_1}, {\bf x}_{j_2}, \dots, {\bf x}_{j_l}$, namely, elements in the elimination ideal:
	$$I_{N}^{\rm elim} ~:=~ I_N \cap ~\Q[ {\boldsymbol \kappa};~{\bf x}_{j_1}, {\bf x}_{j_2}, \dots, {\bf x}_{j_l}]~.$$
When eliminating a single species $X_{\ell}$,
we use the notation:
$$I_N^{ {\rm elim}(x_{\ell} ) } ~:=~ I_N \cap ~\Q[ {\boldsymbol \kappa};~{\bf x}_{1}, {\bf x}_{2}, \dots, {\bf x}_{\ell-1}, {\bf x}_{\ell+1}, {\bf x}_{\ell+2}, \dots, {\bf x}_{n}]~.$$

We consider the following setup: 
a set of observable variables ${\bf x}_{j_1}, {\bf x}_{j_2}, \dots, {\bf x}_{j_l}$, and a network $N$ obtained by gluing two networks $N_1$ and $N_2$ over complexes or reactions.  
We consider the corresponding elimination ideals:
$I_{N_i}^{\rm elim} = I_{N_i} \cap ~\Q[ {\boldsymbol \kappa(i)};~{\bf x}_{j_1}, {\bf x}_{j_2}, \dots, {\bf x}_{j_l}]$,
for $i=1,2$, where ${\boldsymbol \kappa(i)}$ denotes the vector of
rate constants for network $N_i$.

We aim to investigate how 
$I_{N}^{\rm elim}$ is related to $I_{N_1}^{\rm elim}$ and
$I_{N_2}^{\rm elim}$. 
Specifically, when can $I_{N}^{\rm elim}$ be used to
  compute $I_{N_1}^{\rm elim}$ and
$I_{N_2}^{\rm elim}$, and, conversely, when is knowing  $I_{N_1}^{\rm elim}$ and
$I_{N_2}^{\rm elim}$ sufficient for reconstructing  $I_{N}^{\rm elim}$? 
One way we address these questions is by comparing 
$\phi_i(I_N^{\rm elim})$ to $I_{N_i}^{\rm elim}$ (for $i=1,2$),
where $\phi_i$ is the projection to the species variables and rate constants of 
network $N_i$.
More precisely, $\phi_i$ is the ring homomorphism defined on generators as follows:
 \begin{align*}
	\phi_i ~:~ \Q[ {\boldsymbol \kappa}; {\bf x}] ~&\to~ \Q[ {\boldsymbol \kappa}(i) ; {\bf x}(i)] \\
	\notag
	\kappa_a  ~&\mapsto~ 
		\begin{cases}
			\kappa_a & {\rm ~if~} \kappa_a \in {\boldsymbol \kappa}(i) \\
			0 & {\rm ~if~} \kappa_a \notin {\boldsymbol \kappa}(i)
		\end{cases}\\
	\notag
	x_a  ~&\mapsto~ 
		\begin{cases}
			x_a & {\rm ~if~} x_a \in {\bf x}(i)  \\
			0 & {\rm ~if~} x_a \notin {\bf x}(i) 
		\end{cases}		
\end{align*}

\begin{rmk} \label{rmk:ODEs-as-sum}
Recall from Lemma~\ref{lem:glue-ODEs} that each system polynomial $h_j$ of $N$ can be written in the form $h_j=f_j+\widetilde{g}_j$ where $f_j$ is the $j$-th system polynomial of $N_1$ and $\widetilde{g}_j$ is the $j$-th system polynomial of the network obtained from $N_2$ by removing reactions in $N_1$.  
It follows that $\phi_1(h_j)=f_j$. 
\end{rmk}

We will prove the containment
$\phi_i (I_N^{\rm elim})  \subseteq I_{N_i}^{\rm elim}$ 
(Proposition~\ref{prop:set-relationship}), and then investigate when 
the containment is an equality.

\subsection{Results}
We begin by showing that,
before elimination,
 our 
ideals of interest, $\phi_i (I_N^{\rm elim})$ and $I_{N_i}^{\rm elim}$, 
are in fact equal.

\begin{lemma}\label{lem:equality} Let $N$ 
be the reaction network obtained by gluing two networks $N_1$ and $N_2$ over a set of complexes or a set of reactions. 
Then, for $i=1,2$, we have the following equality: 
$$I_{N_i} ~=~ \phi_i(I_N) ~\subseteq~ \Q[ {\boldsymbol \kappa}(i),{\bf x}(i)]~.$$
\end{lemma}

\begin{proof}  
 Let $f_1,f_2, \ldots, f_n$ be the system polynomials of $N_1=(\mathcal S_1, \mathcal C_1, \mathcal R_1) $, and 
  $g_1, g_2,\ldots, g_n$ the system polynomials of $N_2=(\mathcal S_2, \mathcal C_2, \mathcal R_2)$, where
 $f_i:=0$ (respectively, $g_i:=0$)
  for species $i \in \calS_2 \setminus \calS_1$ (respectively, $i \in \calS_1 \setminus \calS_2$). 
Let $h_1,h_2, \ldots, h_n$ be the system polynomials of $N$.

 By symmetry, we may assume $i=1$.  
 %
As $\phi_1$ is surjective, we have 
\[
\phi_1(I_N)=\phi_1( \langle h_1,h_2, \dots, h_n \rangle ) 
 	~=~ \langle \phi_1(h_1), \phi_1(h_2), \dots, \phi_1(h_n) \rangle 
	~=~ \langle f_1, f_2, \dots, f_n \rangle 
	~=~ I_{N_1}~,
	\]
where we also used Remark~\ref{rmk:ODEs-as-sum}.

\end{proof}

\begin{proposition}\label{prop:set-relationship}  Let $N$ be a reaction network obtained by gluing two networks $N_1$ and $N_2$ over a set of complexes or a set of reactions. 
Consider a set of (observable) variables ${\bf x}_{j_1}, {\bf x}_{j_2}, \dots, {\bf x}_{j_l}$. 
For $i=1,2$, we have the following containment:
\begin{equation} \label{eq:subset}
\phi_i( I_N \cap  \Q[{\boldsymbol \kappa}; {\bf x}_{j_1}, {\bf x}_{j_2}, \dots, {\bf x}_{j_l}])
~\subseteq~ 
I_{N_i} \cap \Q[{\boldsymbol \kappa}(i); {\bf x}_{j_1}, {\bf x}_{j_2}, \dots, {\bf x}_{j_l}]
~.
\end{equation}
In other words, 
$ \phi_i(I_N^{\rm elim})~ 
 \subseteq~
I_{N_i}^{\rm elim}~.
$ 
\end{proposition}

\begin{proof}
Let $h \in  I_N \cap  \Q[{\boldsymbol \kappa}; {\bf x}_{j_1}, {\bf x}_{j_2}, \dots, {\bf x}_{j_l}]$.  We must show $\phi_i(h) \in I_{N_i} \cap \Q[{\boldsymbol \kappa}(i);{\bf x}_{j_1}, {\bf x}_{j_2}, \dots, {\bf x}_{j_l} ] $.
To see this, first note that $\phi_i(h) \in \phi_i(I_N)= I_{N_i}$, by Lemma~\ref{lem:equality}.
Also, $h \in \Q[{\boldsymbol \kappa}; {\bf x}_{j_1}, {\bf x}_{j_2}, \dots, {\bf x}_{j_l}]$ implies that 
$\phi_i(h) \in 
\Q[{\boldsymbol \kappa}(i);{\bf x}_{j_1}, {\bf x}_{j_2}, \dots, {\bf x}_{j_l} ] $, and this completes the proof.
\end{proof}

Here we give two counterexamples to equality of the containment~\eqref{eq:subset} in Proposition \ref{prop:set-relationship}.
\begin{eg}[Gluing over complexes] \label{ex:counterex-cpx}
Consider the networks $N_1 = \{X_1 \overset{\kappa_{1}}\rightarrow X_2\}$ and 
$N_2 =\{ X_2 \overset{\kappa_{2}}\rightarrow X_1\}$. 
Then by gluing over complexes, we obtain $N = N_1 \cup N_2 = \{ X_1 \overset{\kappa_{2}}{\underset{\kappa_{1}}\leftrightarrows} X_2 \}$. The corresponding steady-state ideals are:
$$I_N = \langle -\kappa_1x_1 + \kappa_2x_2 \rangle , \quad I_{N_1} = \langle \kappa_1x_1 \rangle,  \quad I_{N_2} = \langle \kappa_2x_2 \rangle~.$$
Elimination of $x_1$ gives:
$$I^{\text{elim}(x_1)}_N = \langle 0 \rangle , \quad I^{\text{elim}(x_1)}_{N_1} = \langle 0 \rangle ,  \quad I^{\text{elim}(x_1)}_{N_2} = \langle \kappa_2x_2 \rangle~.$$
Then 
$\phi_2 \left( I^{\text{elim}(x_1)}_N \right) = \langle 0 \rangle \subsetneq 
I^{\text{elim}(x_1)}_{N_2}$, so the containment~\eqref{eq:subset} 
in general is not an equality for gluing over complexes.
\end{eg}

\begin{eg}[Gluing over reactions] \label{ex:counterex-rxn}
Let $N_1=\{ X_3 \overset{\kappa_{1}}\rightarrow X_1+X_3,~  X_4 \overset{\kappa_{2}}\rightarrow X_2\}$ and 
$N_2=\{ X_4 \overset{\kappa_{2}}\rightarrow X_2,~  X_2 \overset{\kappa_{3}}\rightarrow X_1+X_2\}.$ 
Gluing over the reaction $X_4 \overset{\kappa_{2}}\rightarrow X_2$
 yields 
\begin{align} \label{eq:glued-network-example}
N ~=~ N_1 \cup N_2 ~=~
	 \{X_3 \overset{\kappa_{1}}\rightarrow X_1+X_3,~ \quad X_4 \overset{\kappa_{2}}\rightarrow X_2, \quad X_2 \overset{\kappa_{3}}\rightarrow X_1+X_2 \}.
\end{align}
The corresponding steady-state ideals are:
$$I_N = \langle \kappa_1x_3 + \kappa_3x_2, \kappa_2x_4 \rangle , \quad 
	I_{N_1} = \langle \kappa_2x_4,\kappa_1x_3 \rangle,  \quad
	 I_{N_2} = \langle \kappa_2x_4,\kappa_3x_2 \rangle.$$
Elimination of $x_3$ gives:
$$I^{\text{elim}(x_3)}_N = \langle \kappa_2x_4 \rangle , \quad I^{\text{elim}(x_3)}_{N_1} = \langle \kappa_2x_4 \rangle,  \quad I^{\text{elim}(x_3)}_{N_2} = \langle \kappa_2x_4,\kappa_3x_2 \rangle.$$
Again, we find 
$\phi_2 \left( I^{\text{elim}(x_3)}_N \right) = \langle \kappa_2 x_4 \rangle \subsetneq 
I^{\text{elim}(x_3)}_{N_2}$, 
so equality 
of the containment~\eqref{eq:subset} 
does not hold in general for gluing over reactions.
\end{eg}

These counterexamples prompt the question: \emph{Are there combinatorial conditions on $N$ that guarantee equality of the containment~
$ \phi_i ( I_N^{\rm elim} )
 \subseteq
I_{N_i}^{\rm elim}
$ in~\eqref{eq:subset}?}  
Some positive results in this direction are the focus of the next subsections.

\subsubsection{Monomolecular networks}
In this section, we prove three results for monomolecular networks. Throughout the section, we make the following simplifying assumption:
\begin{center}
{\em monomolecular networks do not involve the zero complex.}
\end{center}
For such a network $N$, the mass-action ODEs (and hence the system polynomials) are
linear in the $x_i$'s and can be written 
in matrix notation as
$ x' = A_{\kappa}^T x~,$
where $A_{\kappa}$ is the negative Laplacian of the reaction graph of
$N$.  
Recall that 
 an $n \times n$ Laplacian matrix has rank at most $n-1$. 
 Indeed, the column vectors of a Laplacian matrix always sum to the zero vector. 
 {\em Hence, the system polynomials of a monomolecular network sum to zero,
so we can always delete one system polynomial before generating the steady-state ideal.}
 Our proofs will harness this fact.

\begin{theorem} \label{thm:mono-single-species} 
Let $N$ be a network obtained 
by gluing 
monomolecular networks $N_1$ and $N_2$ 
over a single species, say, $X_k$. 
Then for every species $X_{\ell}$, the following holds for $i=1,2$:
\[
	 \phi_i(I_N^{{ \rm elim}(x_{\ell} ) })~ 
 =~
I_{N_i}^{{ \rm elim}(x_{\ell} )}~.
\]
\end{theorem}

\begin{proof} Let $X_1, \ldots, X_k$ be the species of
    $N_1$, and let $X_k, \ldots, X_n$ be the species of $N_2$. Without
    loss of generality, assume $i=1$. Since
    Proposition~\ref{prop:set-relationship} gives us the inclusion
    $\phi_1(I_N^{ {\rm elim} (x_l)}) \subseteq I_{N_1}^{ {\rm
        elim}(x_l)}$,  we only need to show $I_{N_1}^{ {\rm
        elim}(x_l)} \subseteq \phi_1(I_N^{ {\rm elim} (x_l)})$.  We
    will do this by showing $I_{N_1} \subseteq I_N$.
This suffices, as it is straightforward to check
  that  
$I_{N_1} \subseteq I_N$ implies that $I_{N_1}^{ {\rm elim} (x_l)} \subseteq I_{N}^{ {\rm
        elim}(x_l)}$, and hence 
$I_{N_1}^{ {\rm elim} (x_l)} = \phi_1(I_{N_1}^{ {\rm elim} (x_l)})
\subseteq \phi_1( I_{N}^{ {\rm
        elim}(x_l)}),
$
where the equality follows from the fact that 
$I_{N_1}^{ {\rm elim} (x_l)} \subseteq \mathbb{Q}[\kappa(1);x(1)]$. 

Now let $f_1,f_2,\dots,f_k$ denote the system polynomials of $N_1$.
Since $N_1$ is monomolecular, the sum of all the system polynomials is
$0$, and thus, $I_{N_1}$ is generated by $f_1, \ldots, f_{k-1}$.
Since $N_2$ doesn't contain the species $X_1, \ldots, X_{k-1}$,  the
polynomials $f_1, \ldots, f_{k-1}$ are also system polynomials of $N$
(recall Remark~\ref{rmk:ODEs-as-sum}). Therefore, $I_{N_1} \subseteq I_N$, and so,  $I_{N_1}^{ {\rm elim}(x_l)} \subseteq \phi_1(I_N^{ {\rm elim} (x_l)})$.
\end{proof}

\begin{rmk}
Theorem~\ref{thm:mono-single-species} 
concerns monomolecular networks, so the invariants obtained by Gaussian elimination in the proof are the \emph{type 1 complex-linear invariants} from \citep{karp}.
\end{rmk}

\begin{theorem} \label{thm:mono-reaction}
Let $N$ be obtained by gluing two monomolecular networks $N_1$ and $N_2$
over a single reaction  $X_{j_1} \to X_{j_2}$ or over a pair of 
reversible reactions $X_{j_1} \rightleftharpoons X_{j_2}$.
If $X_{j_1}$ does not belong to any other reaction in $N_2$, and $X_{j_2}$ does not belong to any other reaction in $N_1$, then for every species $X_{\ell}$, the following holds for $i=1,2$:
\[
	 \phi_i(I_N^{{ \rm elim}(x_{\ell} ) })~ 
 =~
I_{N_i}^{{ \rm elim}(x_{\ell} )}~.
\]
\end{theorem}

\begin{proof}
Let $N$ be obtained by gluing monomolecular networks $N_1$ and $N_2$ over a single reaction  $X_{j_1} \to X_{j_2}$ or over a pair of 
reversible reactions $X_{j_1} \rightleftharpoons X_{j_2}$.  Let the species set of $N_1$  be $\mathcal S_1 = \{ X_1, \ldots, X_{j_1 - 1}, X_{j_1}, X_{j_2}\}$, and let the species set of $N_2$ be $\mathcal S_2 = \{ X_{j_1}, X_{j_2}, X_{j_1+1}, \ldots, X_{n}\}$.  Let $i = 1$.  As explained in the proof of Theorem \ref{thm:mono-single-species}, it is enough to show $I_{N_1} \subseteq I_{N}$. Let $f_1, \ldots, f_{j_1 - 1}, f_{j_1}, f_{j_2}$ denote the system polynomials of $N_1$. Since $N_1$ is monomolecular, $I_{N_1}$ is generated by $f_1, \ldots, f_{j_1 - 1}, f_{j_1}$.  Since every reaction in $N$ involving the species $X_1, \ldots, X_{j_1 - 1}, X_{j_1}$ appears in $N_1$, it follows that $f_1, \ldots, f_{j_1 - 1}, f_{j_1}$ are system polynomials for $N$.  Hence, $I_{N_1} \subseteq I_N$.  When $i=2$, a similar argument can be applied,
as all the reactions in $N$ involving the species $X_{j_2}, X_{j_2 +1}, \ldots, X_{n}$ appear in $N_2$.
\end{proof}

The following result concerns networks for which we can use the elimination ideals of $N_1$ and $N_2$ to directly compute the elimination ideal of $N$.

\begin{theorem} \label{thm:mono-glue-species-invt}
Let $N$ be obtained by gluing two monomolecular networks $N_1$ and $N_2$
over a single species, say, $X_k$.  
If the flow through $X_k$ is unidirectional 
(i.e., whenever $X_k$ is the product of a reaction, the reactant is in $N_1$, and whenever $X_k$ is the reactant, the product is in $N_2$; or vice-versa), 
then,  for every species $X_{\ell}$,
\[
	 I_N^{{ \rm elim}(x_{ \ell} ) }~ 
 =~
I_{N_1}^{{ \rm elim}(  x_{\ell} )}  +
I_{N_2}^{{ \rm elim}( x_{\ell} )}~.
\]
\end{theorem}

\begin{proof}  
We may assume that all reactions to 
$X_k$ are from $N_1$ and all 
reactions from $X_k$ 
are towards $N_2$.  
Assume the species of $N_1$ are $X_1, X_2, \dots, X_k$, and those of
$N_2$ are $X_k, X_{k+1}, \dots, X_s$.  
Let $f_1,f_2,\dots,f_k$ denote the system polynomials of $N_1$; and
let $g_k,g_{k+1}, \dots, g_s$ denote those of $N_2$.
Then the system polynomials of N are $f_1,f_2,\dots,f_{k-1}, f_k+g_k ,
g_{k+1},\dots g_s$.  
For monomolecular networks, the sum of all
system polynomials is 0.  Thus, we can delete one polynomial (here, the $k$-th) from
those of $N$, and still generate the steady-state ideal of $N$ -- and
similarly for $N_1$ and $N_2$:
\[
I_N ~=~ \langle f_1,…,f_{k-1}, g_{k+1},…g_s \rangle ~=~\langle
f_1,f_2,\dots, f_{k-1}\rangle + \langle g_{k+1},g_{k+2}, \dots g_s \rangle~=~I_{N_1} +I_{N_2}~.
\]
The variables (the $x_i$'s and $\kappa_{ij}$'s) in $f_1,\dots,
f_{k-1}$ are
disjoint from those in $g_{k+1}, \dots, g_s$. 
So, a Gr\"obner basis (with respect to an ordering for eliminating
$x_{\ell}$) of $I_N$, 
is obtained by taking the union of Gr\"obner bases (with respect to
the same ordering)  of each ideal $N_1$ and $N_2$.  
The reason for this is, in Buchberger's algorithm, we need only take
s-pairs of two polynomials with leading terms with at least 1
variable in common~\cite[Chapter 2.9]{CLO}.  Thus, we are done.
\end{proof}

\subsubsection{Beyond monomolecular networks}
In the future, we hope to generalize results we proved for monomolecular networks to the non-monomolecular setting.
Specifically, we pose the following problem:

\begin{prob} \label{prob:equality}
Find conditions that guarantee the equality 
$ \phi_i ( I_N^{\rm elim} ) = I_{N_i}^{\rm elim}$. 
\end{prob}

We end this subsection with two examples involving non-monomolecular networks, which
 may point the way toward 
progress on Problem~\ref{prob:equality}.
\begin{eg} 
We revisit, from Example~\ref{ex:counterex-rxn}, the networks 
	$N_1 = \{X_3 \overset{\kappa_{1}}\rightarrow X_1+X_3, ~ X_4 \overset{\kappa_{2}}\rightarrow X_2\}$ and 
$N_2 =\{ X_4 \overset{\kappa_{2}}\rightarrow X_2, ~ X_2 \overset{\kappa_{3}}\rightarrow X_1+X_2\}$. 
Gluing over the shared reaction $ X_4
\overset{\kappa_{2}}\rightarrow X_2$, we obtain 
the network $N=N_1 \cup N_2$ shown in~\eqref{eq:glued-network-example}.
Recall that the corresponding steady-state ideals are:
$$I_N = \langle \kappa_1x_3 + \kappa_3x_2, \kappa_2x_4 \rangle , \quad
I_{N_1} = \langle \kappa_2x_4,\kappa_1x_3 \rangle,  \quad I_{N_2} =
\langle \kappa_2x_4,\kappa_3x_2 \rangle.$$
In Example~\ref{ex:counterex-rxn}, we eliminated $x_3$; here we
instead eliminate $x_4$, which gives:
$$I^{\text{elim}(x_4)}_N = \langle \kappa_1x_3 + \kappa_3x_2 \rangle , \quad I^{\text{elim}(x_4)}_{N_1} = \langle \kappa_1x_3 \rangle,  \quad I^{\text{elim}(x_4)}_{N_2} = \langle \kappa_3x_2 \rangle.$$
Notice that (for $i=1,2$) we have the equality 
$\phi_i ( I^{\text{elim}(x_4)}_N ) =  
I^{\text{elim}(x_4)}_{N_i}$.
\end{eg}

\begin{eg}[Phosphorylation] \label{ex:phos}
Protein modification plays a crucial role in protein activation and de-activation. Generally, an enzyme binds to a substrate, forms an enzyme-substrate complex, and then modifies the substrate by adding, for instance, a phosphate (phosphorylation) or removing one (dephosphorylation). 
Consider two one-site phosphorylation cycles $N_1=\{S_0+E \leftrightarrows X \rightarrow S_1+E, \ S_1 + F \leftrightarrows Y \rightarrow S_0 +F\}$ and $N_2=\{S_1+E \leftrightarrows X_1 \rightarrow S_2+E, \ S_2 + F \leftrightarrows Y_2 \rightarrow S_1 +F\}$. Identifying the 
shared complexes $S_1+E$ and also $S_1+F$ in each of the networks and gluing over them, we obtain $N=N_1 \cup N_2$, a two-site phosphorylation cycle \citep{Enzyme-sharing}. 
For every species $j$ of $N$ and for both networks, i.e., $i=1,2$,
we have 
$\phi_i \left( I_N^{\rm elim(x_j)} \right) = I_{N_i}^{\rm elim(x_j)}$.  This result is surprising, and it prompts us to ask, {\em For which protein modification networks does the equality  
$ \phi_i ( I_N^{\rm elim} ) = I_{N_i}^{\rm elim}$ hold?}
\end{eg}

\subsection{Discussion}
Decomposition results like the ones in this section 
are a common theme in algebraic statistics and phylogenetic algebraic geometry \citep{AR08, DSS09, EKS14}, and thus
one of our aims is to deepen the interaction between the fields of algebraic statistics and algebraic systems biology.  
A guiding question for the future therefore is as follows: 
Can we use techniques from algebraic statistics to analyze the steady-state invariants in larger 
classes of models? 

Additionally, we hope that our results set the stage for obtaining more than just steady-state invariants.  
Specifically, just as elimination techniques helped build a framework for understanding a network's capacity for multistationarity (multiple steady states)~\citep{FeliuWiuf}, in the future our results may also contribute to understanding 
this topic, which we turn to next.


\section{Multistationarity} \label{sec:mss}
For a network with a decomposition into two subnetworks, the previous sections related its identifiability properties and steady-state invariants to that of the two subnetworks.  
Now we turn to a third topic, multistationarity, and show through several examples that this property is sometimes preserved and sometimes lost when going from a subnetwork to a network.  

\subsection{Background} \label{subsec:steady-states}

Recall that a {\em steady state} of a reaction kinetics system is a nonnegative concentration vector $x^* \in \Rnn^n$ at which the ODEs~\eqref{eq:ODE}  vanish.  
We are interested in networks that admit multiple steady states, and if so whether these multiple positive states are stable (i.e., accessible).  This is of particular biological importance for cellular decision making. If a system has two positive steady states, but only one is ever stable, the system cannot choose between states, for example, cell fate.

\begin{definition} \label{def:steady-ste} ~
\begin{enumerate}
	\item A steady state $x^*$ is {\em nondegenerate} if ${\rm Im}\left( df_{\kappa} (x^*)|_{\St} \right) = \St$.  (Here, $df_{\kappa}(x^*)$ is the Jacobian matrix of $f_{\kappa}$ at $x^*$.)  
	\item A nondegenerate steady state is 
{\em exponentially stable} if each of the $\sigma:= \dim(\St)$ nonzero eigenvalues of $df_{\kappa}(x^*)$ has negative real part. 
\end{enumerate}
\end{definition}
\noindent
Also, we distinguish between {\em positive steady states} $x ^* \in \mathbb{R}^n_{> 0}$ and {\em boundary steady
states} $x^* \in  \left( \mathbb{R}^n_{\geq 0} \setminus \mathbb{R}^n_{> 0} \right)$.  

\begin{definition} \label{def:mss}
 ~
\begin{enumerate}
	\item A reaction network is {\em multistationary} if, for some choice of positive rate constants $\kappa_{ij}$, 
the resulting mass-action kinetics system~\eqref{eq:ODE-mass-action} admits
two or more positive steady states in some stoichiometric compatibility class~\eqref{eqn:invtPoly}. Otherwise, the network is {\em monostationary}.
	\item Analogously, a network is {\em nondegenerately multistationary} or {\em multistable} if it admits multiple nondegenerate or exponentially stable, respectively, positive steady states.
\end{enumerate}
\end{definition}

\subsection{Monomolecular networks are {\em not} nondegenerately multistationary} \label{subsec:mono-mss}
We begin by showing that monomolecular networks are {\em not} nondegenerately multistationary.
\begin{proposition} \label{prop:mono-mss}
If $G$ is a reaction network in which each reactant complex is either monomolecular or the zero complex, then $G$ is not nondegenerately multistationary.
\end{proposition}

\begin{proof}
Let $G$ be a network in which all nonzero reactants are monomolecular.  Let $\invtPoly=(x^0+\St) \cap \mathbb{R}^s_{\geq 0}$ be a stoichiometric compatibility class of $G$, and let \{$\kappa_{ij}$\} be any choice of positive rate constants.  We must show that the resulting system does {\em not} admit more than one nondegenerate positive steady state in $\invtPoly$.  The steady states in $\invtPoly$ are the solutions of the system comprising the following equations:
 	\begin{enumerate}
	\item the equations obtained by setting all right-hand sides of the ODEs to zero (these are linear because the reactants of $G$ are at-most-monomolecular), and
	\item the linear equations $\langle x-x^0, ~ v(i) \rangle = 0$, where $v(1),v(2),\dots, v(T)$ form a basis of $S^\perp$. 
	\end{enumerate}
Thus, the steady states in $\invtPoly$ form the solution set of a system of linear equations; hence there are 0, 1, or infinitely many.  If there are infinitely many, then the set of steady states in $\invtPoly$ is a positive-dimensional affine subset of $\invtPoly$, and so every steady state in $\invtPoly$ is degenerate.  
\end{proof}

\begin{rmk} \label{rmk:mono-nondeg-mss}
We can not remove `nondegenerately' from the statement of Proposition~\ref{prop:mono-mss}.  
This fact was illustrated by \citet{JS} via the network $G = \{ 0 \leftarrow A \to 2A \}$.  Its only reactant, $A$, is monomolecular.  If the two rate constants are equal, then every positive value of $x_A$ is a degenerate steady state.  
When the two rate constants differ, then the resulting system admits no positive steady states.  
Thus, $G$ is multistationary, but only degenerately so.
 \end{rmk}

\subsection{``Lifting'' multistationarity from subnetworks and other networks} \label{subsec:lift}
When can we ``lift'' multiple steady states from a subnetwork to the full network?  That is, from simply knowing that a subnetwork (or other related network) is multistationary, when can we conclude that the full network is, too?
Investigating this question is currently an active area of research.
A typical result in this area, described informally, is as follows: 
if $N$ is a subnetwork of $G$ 
and both networks contain all possible flow reactions, then if $N$ is multistationary then $G$ is as well~\citep{joshi2012atoms}.  Another is the following: if $N$ is obtained from $G$ by removing ``intermediate'' complexes, 
then if $N$ is multistationary then $G$ is too~\citep{FeliuWiuf}. A survey of these types of results is in~\cite[\S 4]{mss-review}, and additional results appear in recent work by \citet{BP-inher}.

We end this subsection with a cautionary example, which illustrates why results in this area are nontrivial.  If a subnetwork of a given network is multistable, it is tempting to conclude that the larger network is as well.  As explained above, in some cases we have results that guarantee that this will work, but this does not hold in general: 

\begin{eg}[Having a multistable subnetwork does {\em not} imply multistability.]
The following network is multistable~\citep{joshi2013complete}:
\[
0 \leftrightarrow A
\quad \quad
2A \leftrightarrow 3A~.
\]
However, adding the reaction $A \to B$ to the network  yields a network with no positive steady states (for any choice of rate constants).
Indeed, the concentration of $B$ goes to $\infty$.
\end{eg}

The main question guiding the remainder of this section is: {\em for two networks $N_1$ and $N_2$ that are joined together in some way, how is the capacity for multistationarity 
of the overall network related to that of $N_1$ and $N_2$?}  We are interested in two ways of joining the networks: 
adding a single reaction from $N_1$ to $N_2$ (Section~\ref{sec:mss-add-rxn}), and 
``gluing'' over a (unique) complex that is common to both $N_1$ and $N_2$ (Section~\ref{sec:mss-glue}). 

\subsection{Joining two networks by a new reaction} \label{sec:mss-add-rxn}
We show by example that by joining multistationary networks $N_1$ and $N_2$ (with no complex in common) by a new reaction (from a complex in $N_1$ to one in $N_2$), the new network may be non-multistationary or multistationary.

\begin{eg}[Resulting network is {\em not} multistationary]
The idea behind this example is the following: if we add a new reaction to join one multistationary network $N_1$ to another one $N_2$, then if both networks are mass-preserving and their respective species sets are disjoint, then the new network ``drains'' all species concentrations from $N_1$ and hence no positive steady states exist.  Concretely, let $N_1 = \{ 3A \leftrightarrows 2A+B,~A+2B \leftrightarrows 3B\}$, and let 
 $N_2=\{ 3C \leftrightarrows 2C+D,~ C+2D \leftrightarrows 3D\}$.  Clearly, the two networks are equivalent.  Each network is multistationary (multistable, in fact~\citep{Smallest}).  However, adding the reaction $3A \to 3 C$ to join the two networks yields a network with no positive steady states (for any choice of reaction rate constants).
\end{eg}

\begin{eg}[Resulting network is multistationary]
Let $N_1 = \{0 \leftarrow A, ~2A \to 3A \leftarrow 4A\}$, 
and let $N_2 = \{ 5A \leftarrow 6A,~7A \to 8A \leftarrow 9A\}$.  Each network $N_i$ admits 2 positive steady states~\cite[\S 3]{JS}.  Adding the reaction $4A \to 5A$ to join the two networks yields a network that admits 5 positive steady states~\citep{JS}.
\end{eg}

\begin{eg}[Resulting network is multistationary, even if species sets of $N_1$ and $N_2$ are disjoint]
Let $N_1=\{0 \leftrightarrows A,~2A \to 3A\}$, and let $N_2 =\{B \leftrightarrows 2B,~3B \to 4B\}$.  Each network $N_i$ admits 2 positive steady states~~\cite[\S 3]{JS}.  Adding the reaction $A \to B$ to join the two networks yields a network that admits 4 positive steady states (networks $N_1$ and $N_2$ are decoupled, so the maximum number of positive steady states multiplies).
\end{eg}


\subsection{Joining two networks by gluing over a complex} \label{sec:mss-glue}

The following examples show that by joining two multistationary networks $N_1$ and $N_2$ by a single shared complex, the resulting network may be non-multistationary or multistationary.

\begin{eg}[Resulting network is {\em not} multistationary]
Let $N_1=\{ 0 \leftarrow A+B,~3A \to 4A+B\}$, and let $N_2=\{A+B \to 2A,~2A+3B \leftarrow 3A+2B\}$.  Each network $N_i$ admits multiple positive steady states~\citep{JS}.  Gluing the  networks over the unique shared complex, $A+B$, yields a network that (it is easy to check) always has a unique positive steady state.
\end{eg}

\begin{eg}[Resulting network is multistationary]
Let $N_1=\{0 \leftrightarrows A, ~2A \to 3A\}$, and let $N_2=\{3A \leftarrow 4A,~5A \leftrightarrows 6A\}$.  Each network $N_i$ admits 2 positive steady states~\citep{JS}.  Gluing the two networks over the unique shared complex $3A$ yields a network that admits 5 positive steady states~\citep{JS}.
\end{eg}

The above examples motivate some problems for future work.  
\begin{prob}
Formulate necessary or sufficient conditions under which two multistationary networks, when joined by a new reaction or glued over a complex, yield another multistationary network.
\end{prob}
We are also interested in obtaining a B\'ezout-type upper bound on the maximum number of positive steady states arising when two networks are joined.  Specifically, if $N_1$ admits $m_1$ positive steady states, and $N_2$ admits $m_2$, does it follow that the joined network admits at most $m_1 m_2$ positive steady states?
Finally, the biological interest goes beyond multistationarity, to multistability, so we ask, {\em when does joining two multistable networks yield another multistable network?}


\section{Discussion} \label{sec:discussion}
As mentioned earlier, systems biology is in need of theory pertaining to what happens when biological pathways are joined or decomposed.  Accordingly, this work contributes to starting such a theory. 
Our results and examples investigated the effects of joining or decomposing networks on three properties: identifiability, steady-state invariants, and multistationarity.  
Many of our results focused on monomolecular networks, and we also provided initial steps for systems 
with higher molecularity.   
Going forward, the techniques presented in this work could be used to extend our results to more complex systems, such as bimolecular networks, including 
signaling networks such as the so-called MESSI systems~\citep{messi}. 

Another future direction is to extend our results to allow for more ways of joining networks.  For instance, our results on identifiability  pertained only to joining networks by a one-way flow, while our results on steady-state invariants focused on gluing over complexes or reactions.  It would be interesting, therefore, to prove identifiability results for networks obtained by gluing over complexes or reactions, and also steady-state invariants results for networks joined by a one-way flow.
Indeed, this work forms a starting point for understanding fundamental questions about joining and decomposing networks, 
and opens new avenues for tackling more complicated networks.

\subsection*{Acknowledgements}
This project began at a SQuaRE (Structured Quartet Research Ensemble) at AIM, and the authors thank AIM for providing financial support and an excellent working environment.  EG was supported by NSF DMS-1620109. HAH gratefully acknowledges funding from EPSRC Postdoctoral Fellowship (EP/K041096/1) and a Royal Society University Research Fellowship.  NM was partially supported by the Clare Boothe Luce Program from the Luce Foundation.  AS was partially supported by the NSF (DMS-1312473/1513364 and DMS-1752672)
and the Simons Foundation (\#521874).  
The authors thank Alexey Ovchinnikov, Gleb Pogudin, and Peter Thompson
for helpful discussions, and acknowledge two
diligent referees whose helpful suggestions which improved this article.

\bibliographystyle{agsm}
\bibliography{square}


\appendix

\section{Proof of Proposition~\ref{prop:newy}} \label{sec:appendix}
We prove Proposition~\ref{prop:newy}, which we restate here in the language of compartmental models:

\begin{prop} \label{prop:newy-restated}
Consider a linear compartmental model $\mathcal M = (\mathfrak G, In, Out, Leak)$, with $\mathfrak{G} = (V,E)$.
Assume that there exists a compartment $i$ such that the 
 {output-reachable subgraph to} $i$ is $\mathfrak G$.
Then for every $j \in V \setminus \{i\}$, 
there exists an equation of the form $x_j=g$
that holds (for generic values of the parameters $a_{kl}$) 
along all solutions of $\mathcal M$, where 
$g$ is a 
$\mathbb{Q}(a_{kl})$-linear combination 
of the variable $x_i$ and the input variables $u_p$ (for $p \in In$)
and their derivatives $x_i^{(q)}$ and $u_p^{(q)}$, and 
the coefficient of at least one of the $x_i^{(q)}$'s is nonzero.
\end{prop}
\noindent
We first need Lemma~\ref{lem:B} below, which requires several definitions.  
A {\em directed 0-tree} $T$ on vertices $\{0,1,\dots, n-1\}$ is a directed graph such that the underlying undirected graph is cycle-free and for every $j=1,\dots,n-1$ there is a directed path $j \to \dots \to 0$ in $T$ from $j$ to $0$.  
A {\em walk} in a directed graph is a sequence of edges $i_1\to i_2 \to \dots \to i_k$ (repeated edges allowed).  If $W$ is a walk in an edge-labeled directed graph, then $a^W$ denotes the product of the edge labels of $W$.

\begin{lemma} \label{lem:B} Let $n \geq 2$.  
Let $T$ be a directed 0-tree on vertices $\{0,1,\dots, n-1\}$ with edges $i \to j$ labeled by $a_{ji}$.  Let $\widetilde T$ be the directed graph obtained from $T$ by adding, for each edge $i\to j$, a self-loop at vertex $i$ labeled by $-a_{ji}$.  Let $\mathfrak B$ denote the $(n-1) \times (n-1)$ matrix where 
	\begin{align} \label{matrix:B}
	{\mathfrak B}_{ij} ~=~ 
		\sum_{ \{ \text{length-$i$ walks $W$ in $\widetilde T$ from j to 0} \}} a^W~.
	\end{align}
Then $\det {\mathfrak B}$, which is a polynomial in $\mathbb{Q}[a_{ji} \mid i \to j \text{ is an edge of } T]$, is nonzero.
\end{lemma}

\begin{proof} 
By construction, the determinant of $\mathfrak B$ is as follows:
\begin{align} \label{eq:det-B}
	\det {\mathfrak B} ~=~ 
		\sum_{\sigma \in S_{n-1}} {\rm sign}(\sigma) \prod_{i=1}^{n-1} \left(
			\sum_{ \{ \text{length-$\sigma(i)$ walks $W$ in $\widetilde T$ from $i$ to 0} \}} a^W
		\right)~.
\end{align}

Reordering vertices of $T$ reorders the columns of $B$, which only multiplies
$\det \mathfrak B$ by $1$ or $-1$.  So, we now reorder the vertices $1,\dots, n-1$ of $T$, so that they are in an order obtained from a breadth-first search (in the underlying undirected graph of $T$) from vertex 0.  In other words, vertices at distance 1 from $0$ come first, then those at distance 2, and so on.  Hence, letting $d(i)$ denote the distance of vertex $i$ from $0$, it follows by construction that $d(i) \leq i$.

For $i=1,\dots, n-1$, 
let $P(i)=(i\to j_1 \to \dots \to j_{d(i)-1} \to 0)$ denote the unique path in $T$ from $i$ to $0$.  Let $W(i)$ denote the length-$i$ walk in $\widetilde T$ obtained by prepending $i-d(i)$ self-loops at $i$ to the path $P(i)$.  The corresponding monomial $a^{W(i)}$ is as follows:
\[
	a^{W(i)} ~=~ 
		(-a_{j_1i})^{i-d(i)} a_{j_1 i} a_{j_2 j_1} \dots a_{j_{d(i)-i} j_{d(i)-i-1}} a_{0,j_{d(i)-i}}~.
\]

It follows that the following monomial is in the expansion of $\det \mathfrak B$:
	\[
	M~=~ 
	a^{W(1)} a^{W(2)} \dots a^{W(n-1)}~.
	\]
Specifically, this monomial is part of the summand in~\eqref{eq:det-B} where $\sigma$ is the identity permutation.	
	
Hence, to show that $\det \mathfrak B$ is nonzero, it suffices to show the following:\\
\noindent
{\bf Claim}: There is no other set of walks $\{ Q(1), \dots, Q(n-1) \}$, such that there exists a permutation $\tau \in S_{n-1}$ such that (for $i=1,\dots,n-1$) $Q(i)$ is a length-$\tau(i)$ walk in $\widetilde T$ from $i$ to $0$, and for which $M= \pm 	a^{Q(1)} a^{Q(2)} \dots a^{Q(n-1)}$.

We prove this claim by induction on $n$, the number of vertices in $T$.
In the base case, when $n=2$, there is a unique walk (namely, $1\to 0$) of length 1 from vertex $1$ to $0$.

For the inductive step, assume that for directed $0$-trees on $(n-1)$ vertices that are ``breadth-first-search ordered'' (as explained above), the claim is true.  Let $T$, as above, be a
 0-tree on vertices $\{0,1,\dots, n-1\}$, and also let  
$\widetilde T$ and $M$ be as above.  Assume that $M= \pm a^{Q(1)} a^{Q(2)} \dots a^{Q(n-1)}$, as in the claim.  We must show that $W(i)=Q(i)$ for all $i=1,\dots, n-1$.

Consider the vertex $n-1$, and denote the unique path in $T$ from $n-1$ to $0$ by $n-1 \to j_1 \to j_2 \dots \to j_{d(n-1)-1} \to 0$.  
By the choice of ordering, $n-1$ is a leaf of $\widetilde T$. 
So, $a_{j,n-1}$ divides $a^{W(n-1)}$ and $a^{Q(n-1)}$ but none of the other $a^{W(i)}$'s or $a^{Q(i)}$'s.   
In fact, $a_{j_1,n-1}^{n-d(n-1)}$ divides $a^{W(n-1)}$, by construction of $W(n-1)$ and so 
$a_{j_1,n-1}^{n-d(n-1)}$ also divides $a^{Q(n-1)}$ (here we use the fact that $M= \pm a^{Q(1)} a^{Q(2)} \dots a^{Q(n-1)}$).  
However, $W(n-1)$ is the only walk $W$ in $\widetilde T$ that (1) ends at 0, (2) has length at most $n-1$, and (3) involves enough self-loops at $n-1$ in order for $a_{j_1,n-1}^{n-d(n-1)}$ to divide the corresponding monomial $a^W$.
Thus, $Q(n-1)=W(n-1)$.

Hence, 
$a^{W(1)} a^{W(2)} \dots a^{W(n-2)}~=~\pm 
a^{Q(1)} a^{W(2)} \dots a^{Q(n-2)}$, 
and the corresponding walks $W(i)$ and $Q(i)$ arise from the tree $\widetilde{T}'$ obtained from $\widetilde T$ by deleting the leaf $n-1$.  Notice that the vertices of $\widetilde{T}'$ are ``breadth-first search ordered''.  So, by the inductive hypothesis, $W(1)=Q(1)$, \dots, $W(n-2)=Q(n-2)$.  Hence, the claim holds, and this completes the proof.
\end{proof}

\begin{proof}[Proof of Proposition~\ref{prop:newy-restated}]
Let $n$ denote the number of compartments.  We may assume $n \geq 2$, as otherwise there is nothing to prove.  
By relabeling the compartments, if necessary, we may assume that $i=n$ and the remaining compartments are labeled by $1,2,\dots,n-1$.  

Our proof and notation follow the proof of \cite[Lemma 3]{MeshkatSullivantEisenberg}.  
We write $x' = Ax+u$, where 
 $A$ is the $n \times n$ compartmental matrix, 
 with entries given by: 
\[
  A_{\ell j} 
  ~:=~ \left\{ 
  \begin{array}{l l l}
    -a_{0 \ell}-\sum_{k: \ell \rightarrow k \in E}{a_{k \ell}} & \quad \text{if $\ell = j$ and } \ell \in Leak\\
        -\sum_{k: \ell \rightarrow k \in E}{a_{k \ell}} & \quad \text{if $\ell = j$ and } \ell \notin Leak\\
    a_{\ell j} & \quad \text{if $j\rightarrow{\ell}$ is an edge of $\mathfrak{G}$}\\
    0 & \quad \text{otherwise.}\\
  \end{array} \right.
\]

Let $\widetilde A$ denote the matrix obtained from $A$ by removing row-$n$ and column-$n$. 
Let ${\bf a}$ (respectively, ${\bf b}$) be the row (respectively, column) vector obtained by removing the $n$-th entry from row-$n$ (respectively, column-$n$) of $A$.
Finally, let $\widetilde x:=(x_1,x_2,\dots,x_{n-1})^T$ and $\widetilde u:=(u_1,u_2,\dots,u_{n-1})^T$, where $u_j:=0$ if $j \notin In$.

Let $B$ denote the following $(n-1) \times (n-1)$ matrix: the first row is $\bf{a}$, the second row is ${\bf a} \widetilde A$, the third row is ${\bf a} {\widetilde A}^2$, \dots , and the last row is ${\bf a} {\widetilde A}^{n-2}$.  Consider the following claim: \\ 
{\bf Claim A:} For generic values of the $a_{k \ell}$'s, the matrix $B$ is invertible.

To prove this claim, we must show that $\det B$, which is a polynomial in the $a_{k \ell}$'s, is nonzero.  Relabel the vertex $n$ in $\mathfrak G$ by $0$, and call this graph $\mathfrak G'$.
Let $T$ denote a subgraph of $\mathfrak G'$ that is a directed $0$-tree (such a subgraph exists by the hypothesis of being output-reachable).  
Let $\widetilde T$ be the graph arising from $T$ as defined in Lemma~\ref{lem:B}, and let
$\mathfrak B$ be the matrix~\eqref{matrix:B}.

We claim that $\mathfrak B = B|_{ \{ a_{j \ell }=0 \mid \ell \to j \text{ is  not an edge of } T \}}$.  
To see this, note that $\widetilde A|_{ \{ a_{j \ell}=0 \mid \ell \to j \text{ is  not an edge of } T \}}$ is the adjacency matrix for the graph ${\widetilde T}_0$ obtained by deleting vertex $0$ from $\widetilde T$.  
Hence, the $( i_1,i_2)$ entry in $({\widetilde A})^k|_{ \{ a_{j \ell}=0 \mid \ell \to j \text{ is  not an edge of } T \}}$ is a sum of monomials $a^W$, where the sum is over walks $W$ in ${\widetilde T}_0$ 
of length $k$ from $i_1$ to $i_2$.  The vector ${\bf a}$ encodes the directed edges $\ell \to 0$, and so it is straightforward to check that 
the ${\bf a} ({\widetilde A}^k)|_{ \{ a_{j \ell}=0 \mid \ell \to j \text{ is  not an edge of } T \}}$'s, 
i.e., the rows of $B|_{ \{ a_{j \ell}=0 \mid \ell \to j \text{ is  not an edge of } T \}}$, 
form the matrix 
$\mathfrak B$ as in~\eqref{matrix:B}.

Hence, using Lemma~\ref{lem:B}, we obtain:
\[
\det B|_{ \{ a_{j \ell}=0 \mid \ell \to j \text{ is  not an edge of } T \}} ~=~
\det \mathfrak B ~\neq ~0~.
\]
Hence, $\det B \neq 0$, and so Claim A holds.

As explained in the proof of \cite[Lemma 3]{MeshkatSullivantEisenberg}, solutions to the model $\mathcal M$ satisfy $B \widetilde x = c$, 
where $c$
is the vector of length $n-1$ that decomposes as follows:
\begin{align*}
c~&=~
\begin{pmatrix}
x_n' - A_{nn}x_n - u_n
\\
x_n^{(2)} - A_{nn}x_n' - u_n'- \left( {\bf ab} x_n + {\bf a} \widetilde u \right) 
\\
\vdots
\\
x_n^{(k)} - A_{nn}x_n^{(k-1)} - u_n^{(k-1)} - \sum_{j=0}^{k-2} 
	\left( {\bf a} \widetilde{A}^{k-2-j} {\bf b} x_n^{(j)} + {\bf a} \widetilde{A}^{k-2-j} \widetilde{u}^{(j)} \right) 
\\
\vdots
\end{pmatrix} \\
~&=~
\begin{pmatrix}
x_n' - A_{nn}x_n 
\\
x_n^{(2)} - A_{nn}x_n' - {\bf ab} x_n 
\\
\vdots
\\
x_n^{(k)} - A_{nn}x_n^{(k-1)}  - \sum_{j=0}^{k-2} 
	 {\bf a}\widetilde{A}^{k-2-j} {\bf b} x_n^{(j)}  
\\
\vdots
\end{pmatrix}
-
\begin{pmatrix}
 u_n
\\
 u_n' + {\bf a} \widetilde u  
\\
\vdots
\\
 u_n^{(k-1)} + \sum_{j=0}^{k-2} {\bf a} \widetilde{A}^{k-2-j} 
	 \widetilde{u}^{(j)}  
\\
\vdots
\end{pmatrix}
~=:~
	c^{(x)}+c^{(u)}~,
\end{align*}
where $u_n:=0$ if $n \notin In$.  
Each coordinate of $c$ is a 
$\mathbb{Q}(a_{k l})$-linear combination 
of the variable $x_n$ and the input variables $u_p$ (for $p \in In$)
and their derivatives $x_n^{(q)}$ and $u_p^{(q)}$.
Therefore, as $B$ is invertible (for generic values of the $a_{k l}$'s), then we obtain the desired equations $g_j$:
\[
	\widetilde x ~=~ B^{-1} c ~=:~ (g_1,g_2,\dots, g_{n-1})^T ~,
\] 
once we verify the following claim: \\
{\bf Claim B}: In each $g_{\ell}$, 
the coefficient of at least one of the $x_n^{(q)}$'s is nonzero.

To show this claim, assume for contradiction that, in some $g_{\ell}$, 
the coefficient of every $x_n^{(q)}$ is zero.  Then, by the above decomposition, we obtain $(B^{-1} c^{(x)})_{\ell}=0$ (the zero polynomial).  In other words, letting ${\bf d}$ denote row-$l$ of $B^{-1}$, 
we have $\langle {\bf d},  c^{(x)} \rangle=0$.  

We will show that ${\bf d}$ is the zero vector.   
Among the coordinates $c^{(x)}_{j}$ (for $j=1,\dots, n-1$) of $c^{(x)}$, 
only the last coordinate, 
namely, $c^{(x)}_{n-1}$, contains as a summand $x_n^{(n-1)}$.  
So, in order for  $\langle {\bf d},  c^{(x)} \rangle=0$, we must have that ${\bf d}_{n-1}=0$ (here we use the fact that the coordinates of ${\bf d}$ are in $\mathbb{Q}(a_{kl})$).
Next, let $\widetilde{{\bf d}}$ and $\widetilde{c}^{(x)}$ 
be the vectors obtained by removing the last coordinate from, respectively, ${\bf d}$ and $c^{(x)}$.
We have $\langle \widetilde{{\bf d}},\widetilde{c}^{(x)} \rangle=0$, and so we can apply the same argument as above to obtain that ${\bf d}_{n-2}=0$.  Continuing, we obtain that every coordinate of ${\bf d}$ is zero. 
We have reached a contradiction, and so Claim B holds.  This completes the proof.
\end{proof}


\begin{corollary} \label{cor:alg-obs}
Every output connectable linear compartmental model 
is algebraically observable.
\end{corollary}

\begin{proof}
Consider a linear compartmental model $\mathcal M = (\mathfrak G, In, Out, Leak)$ that is output connectable. 
Let $\ell$ be any compartment.  If $\ell \in Out$, then the state variable $x_{\ell}$ is itself an output variable, and so is already written in terms of output variables.

So, assume that $\ell \notin Out$.  As the model is output connectable, there exists $i \in Out$ such that there is a path from $\ell$ to $i$.  Let $\mathfrak G'$ denote the output-reachable subgraph to $y_i$.  

It is straightforward to check that the restriction of $\mathcal M$ to $\mathfrak G'$ 
(Definition~\ref{def:restrict}) satisfies the hypotheses of
Proposition~\ref{prop:newy-restated} with respect to $i$.  Also, the ODEs of $\mathcal M$ are obtained from those of the restriction by appending the ODEs for state variables $x_j(t)$ with $j$ not in the  vertex set of ${\mathfrak G'}$ (see the proof of \cite[Lemma 3.7]{submodel}).  Thus, the equation $x_{\ell}=g$ obtained from 
Proposition~\ref{prop:newy-restated} expresses $x_{\ell}$ as a function of $x_i$, the inputs, their derivatives, and the parameters.  
Thus, by \citet{DiopWang}, $\mathcal M$ is algebraically observable.  
\end{proof}

\end{document}